\documentclass[11pt,letterpaper]{amsart}
\usepackage{amsfonts}
\usepackage{amsmath}
\allowdisplaybreaks[1]
\usepackage{amssymb}
\usepackage{mathrsfs}
\usepackage{geometry}
\usepackage{graphicx}
\usepackage{subfigure}
\usepackage{hyperref}
\usepackage{microtype}
\usepackage{enumerate}
\usepackage{geometry}
\usepackage{tikz}
\usepackage{tikz-cd}
\usepackage{color}
\numberwithin{equation}{section}

\usepackage{cite}

\newcommand{\al}{\alpha}
\newcommand{\be}{\beta}
\newcommand{\ga}{\gamma}
\newcommand{\Ga}{\Gamma}

\newcommand{\ve}{\varepsilon}

\newcommand{\vp}{\varphi}
\newcommand{\ka}{\kappa}

\newcommand{\R}{\mathbb{R}}

\newcommand{\ccc}{\cdot\cdot\cdot}

\newcommand{\n}[1]{\Vert #1\Vert}

\newcommand{\bbn}[1]{\Big\Vert #1 \Big \Vert}
\newcommand{\lr}[1]{\left\{ #1\right\}}
\newcommand{\lrc}[1]{\left[ #1\right]}
\newcommand{\lrs}[1]{\left( #1\right)}

\newcommand{\lra}[1]{\langle #1\rangle}

\newcommand{\abs}[1]{|#1|}

\newcommand{\bbabs}[1]{\Big | #1 \Big|}
\newcommand{\wt}[1]{\widetilde{#1}}
\newcommand{\wq}{\infty}

\newcommand{\pa}{\partial}

\newcommand{\ol}{\overline}

\begin{document}

\newtheorem{theorem}{Theorem}[section]
\newtheorem{lemma}[theorem]{Lemma}
\newtheorem{proposition}[theorem]{Proposition}
\newtheorem{corollary}[theorem]{Corollary}
\theoremstyle{definition}
\newtheorem{definition}[theorem]{Definition}
\newtheorem{example}[theorem]{Example}
\newtheorem{xca}[theorem]{Exercise}
\theoremstyle{remark}
\newtheorem{remark}[theorem]{Remark}

\numberwithin{equation}{section}

\title[Quantum Mean-field and Semiclassical Limit]{On the mean-field and semiclassical limit from quantum $N$-body dynamics}

\author[X. Chen]{Xuwen Chen}
\address{Department of Mathematics, University of Rochester, Rochester, NY 14627, USA}
\email{xuwenmath@gmail.com}

\author[S. Shen]{Shunlin Shen}
\address{School of Mathematical Sciences, Peking University, Beijing, 100871, China}
\email{slshen@pku.edu.cn}

\author[Z. Zhang]{Zhifei Zhang}
\address{School of Mathematical Sciences, Peking University, Beijing, 100871, China}

\email{zfzhang@math.pku.edu.cn}

\subjclass[2010]{Primary 35Q31, 76N10, 81V70; Secondary 35Q55, 81Q05.}

\date{}

\dedicatory{}

\begin{abstract}
We study the mean-field and semiclassical limit of the quantum many-body dynamics with a repulsive $\delta$-type potential $N^{3\be}V(N^{\be}x)$ and a Coulomb potential, which leads to a macroscopic fluid equation, the Euler-Poisson equation with pressure.
We prove quantitative strong convergence of the quantum mass and momentum densities up to the first blow up time of the limiting equation.
The main ingredient is a functional inequality on the $\delta$-type potential for the almost optimal case $\be\in(0,1)$, for which
we give an analysis of the singular correlation structure between particles.
 \end{abstract}
\keywords{Mean-field Limit, Semiclassical Limit, Compressible Euler Equation, Quantum Many-body Dynamics, Modulated Energy.}
\maketitle
\tableofcontents

\section{Introduction}

\subsection{Background and Problems}
The foundations of microscopic physics are Newton's and
Schr\"{o}dinger equations in the classical and the quantum case respectively.
By the first principle of quantum mechanics, a quantum
system of $N$ particles is described by a wave function satisfying a linear $N$-body Schr\"{o}dinger equation.
In realistic systems like fluids, the particle number is so large that
these $N$-body equations are almost impossible to solve. The macroscopic dynamics are therefore
modeled by phenomenological equations such as the Euler or the Navier-Stokes equations, which are an important part of many areas of pure and applied mathematics, science, and engineering.
These macroscopic equations are usually derived from continuum under ideal assumptions, but they are, in principle,
consequences of the microscopic physical laws of Newton or
Schr\"{o}dinger.
A key goal of statistical mechanics is to justify these macroscopic equations from
microscopic theories in appropriate limit regimes. It is thus of fundamental interest to establish macroscopic equations from the microscopic level.

In the current paper, we start from the bosonic\footnote{N$_2$ and O$_2$ molecules are bosons (99.03\% of air) and 99.05\% H$_{2}$O molecules are bosons.} quantum many-body dynamics with $\delta$-type and Coulomb potentials, and study the mean-field and semiclassical limit which would lead to macroscopic fluid equations as particle number $N$ tends to infinity and Planck's constant $\hbar$ tends to zero.
The dynamics of $N$ quantum particles in 3D are governed by, according to the superposition principle, the
linear $N$-body Schr\"{o}dinger equation:
\begin{align}\label{equ:N-body schroedinger equation}
i\hbar\pa_{t}\psi_{N,\hbar}=H_{N,\hbar}\psi_{N,\hbar}.
\end{align}
Our Hamiltonian $H_{N,\hbar}$ is
\begin{align}\label{equ:hamiltonian}
H_{N,\hbar}=\sum_{j=1}^{N}-\frac{1}{2}\hbar^{2}\Delta_{x_{j}}+\frac{1}{N}\sum_{1\leq j<k\leq N}V_{N}(x_{j}-x_{k})+\frac{\kappa}{N}\sum_{1\leq j<k\leq N}V_{c}(x_{j}-x_{k}),
\end{align}
where the factor $1/N$ is the mean-field averaging factor.
The $\delta$-type and Coulomb potentials are
\begin{equation}
\begin{cases}
&V_{N}(x)=N^{3\be}V(N^{\be}x),\\
 &V_{c}(x)=\frac{1}{|x|},
\end{cases}
\end{equation}
in which, the parameter $\beta\in[0,1]$ characterises different density regimes which correspond to different physical situations.

To have a fixed number of variables in the $N\to \infty$ process, we define
the marginal densities $\ga_{N,\hbar}^{(k)}(t)$ associated
with $\psi_{N,\hbar}(t)=e^{itH_{N,\hbar}}\psi_{N,\hbar}(0)$ in kernel form by
\begin{align}
\ga_{N,\hbar}^{(k)}(t,X_{k},X_{k}')=\int \psi_{N,\hbar}(t,X_{k},X_{N-k})\ol{\psi_{N,\hbar}}(t,X_{k}',X_{N-k})dX_{N-k},
\end{align}
where $X_{k}=(x_{1},...,x_{k})\in \R^{3k}$ and $X_{N-k}=(x_{k+1},...,x_{N})\in \R^{3(N-k)}$. It is believed that nonlinear Schr\"{o}dinger equations (NLS) is the mean-field
limit equation for these quantum $N$-body dynamics, that is,
\begin{align}
\ga_{N,\hbar}^{(1)}(t,x_{1},x_{1}')\sim |\phi(t)\rangle \langle \phi(t)|,
\end{align}
where $\phi(t)$ solves NLS.

There is a large amount of literature devoted to the mean-field theory from quantum many-body dynamics, such as \cite{AGT07,BdOS15,BCS17,BS19,CHPS15,CP10,CP11,CP13,CP14,CP14derivation,CPT10,Che12,Che13,CH13,CH16on,CH16correlation,CH16focusing,CH17,CH19,CH22quantitative,
CH22unconditional,CSZ22,CS14,EY01,ESY06,ESY07,ESY09,ESY10,FKS09,GMM10,GMM11,GM13,GM17,GSS14,HS16,HS19,HTX15,HTX16,KP10,KSS11,KM08,RS09,
She21,Soh15,Soh16,SS15,Xie15}.
In particular, for the case of defocusing $\delta$-type potential, it was Erd{\"o}s, Schlein, and Yau who first rigorously derived the 3D cubic defocusing NLS from quantum many-body dynamics in their groundbreaking papers \cite{ESY06,ESY07,ESY09,ESY10}.\footnote{Around the same time, see also \cite{AGT07} for 1D case.}In their analysis, apart from the uniqueness of the infinite hierarchy which was widely regarded as the most involved part,
understanding the singular correlation structure generated by the $\delta$-type potential was one of the main challenges.

In the mean-field limit as the particle number $N$ tends to infinity, the potential $V_{N}$ converges formally to the Dirac-delta interaction $(\int V)\delta$, also called the Fermi potential. For $\be<\frac{1}{3}$, the average distance between the particles, which is $O(N^{-\frac{1}{3}})$, is much less than the range of the interaction potential, which
is $O(N^{-\beta})$,
and there are many but weak correlations. For $\beta>\frac{1}{3}$, then the analysis is much more involved
because of the strong correlations between particles.
For $\beta$ close or equal to $1$, as the scaling is starting to match the Laplacian operator, it is expected that the $\delta$-type potential generates an interparticle singular correlation structure, closely related to the zero-energy scattering equation
\begin{equation}\label{equ:scattering equation,appendix}
	\left\{
	\begin{aligned}
		&\lrs{-\hbar^{2}\Delta+\frac{1}{N}V_{N}(x)}f_{N,\hbar}(x)=0,\\
		&\lim_{|x|\to \infty} f_{N,\hbar}(x)=1.
	\end{aligned}
	\right.
\end{equation}
The scattering function $(1-f_{N,\hbar}(x))$ varies effectively on the short scale for $|x|\lesssim N^{-\be}$ and has the same singularity as the Coulomb potential at infinity.
It is believed that\footnote{See for example \cite{LSSY05} for the static case and \cite{ESY06,ESY09,ESY10} for the time-dependent case.}, instead of the factorization property, that is,
\begin{align*}
\ga_{N,\hbar}^{(2)}(t,x_{1},x_{2};x_{1}',x_{2}')\sim \ga_{N,\hbar}^{(1)}(t,x_{1};x_{1}')\ga_{N,\hbar}^{(1)}(t,x_{2};x_{2}'),
\end{align*}
the marginal densities should be considered as
\begin{align*}
\ga_{N,\hbar}^{(2)}(t,x_{1},x_{2};x_{1}',x_{2}')\sim f_{N,\hbar}(x_{1}-x_{2})f_{N,\hbar}(x_{1}'-x_{2}')\ga_{N,\hbar}^{(1)}(t,x_{1};x_{1}')\ga_{N,\hbar}^{(1)}(t,x_{2};x_{2}').
\end{align*}
The singular correlation structure is very subtle and plays a crucial role in the mean-field limit from quantum $N$-body dynamics, as it gives an $O(N^{\beta})$ correction to the $N$-body energy.

For the semiclassical limit, the connection between Schr\"{o}dinger-type equations and the classical fluid mechanics was already noted in 1927 by Madelung \cite{Mad27}. Starting from a single NLS, the asymptotic behavior of the wave function as the
Planck's constant goes to zero is studied by many authors using various approaches based on Madelung's fluid
mechanical formulation. See, for example, \cite{Gre98,JLM99,LZ06,Zha02}. For a more detailed survey related to semiclassical limits, see \cite{Car08,Zha08} and references within. There are many deep problems on the study of classical limiting dynamics
from quantum equations.

The joint mean-field and semiclassical limit from quantum $N$-body dynamics formally gives a direct connection between
quantum microscopic systems and classical macroscopic fluid equations. Providing a rigorous proof is certainly a challenging problem.
For the repulsive Coulomb potential, Golse and Paul \cite{GP21}, based on Serfaty's inequality
\cite[Proposition 1.1]{Ser20}, justified the weak convergence to pressureless Euler-Poisson in the mean-field and semiclassical limit. For the case of the $\delta$-type potential, in our previous work \cite{CSWZ23}, we derived the compressible Euler equations with strong and quantitative convergence rate from quantum many-body dynamics by a new strategy of combining the accuracy of the hierarchy method and the flexibility of the modulated energy method.
Subsequently, such a scheme was adopted in \cite{CSZ23} to obtain the quantitative convergence rate from quantum many-body dynamics to the pressureless Euler-Poisson equation.

Despite a series of progress on the mean-field and semiclassical limit from the quantum $N$-body dynamics with singular potentials, a number of challenges remain open:
\begin{enumerate}
\item The derivation of the full Euler-Poisson equation with pressure from the quantum many-body dynamics. In \cite{CSZ23,GP21}, the limiting Euler-Poisson equation is pressureless. However, the pressure is a fluid defining feature and essential for the macroscopic fluid equation. It is thus a fundamental question to understand the emergence of pressure from the microscopic level.
\item The large $\beta$ problem is known to be difficult in the mean-field and semiclassical limit due to the strong correlations between particles. The main challenge lies in the analysis of the singular correlation structure generated by the $\delta$-type potential.
\item To obtain the quantitative strong convergence rate, a double-exponential restriction between $N$ and $\hbar$ was needed in \cite{CSWZ23}, which is of course not optimal.
From the perspective of energy, the restriction should be at least polynomial. To relax the double-exponential restriction, it requires  new and finer techniques.
\item The scheme in \cite{CSWZ23} currently cannot deal with the $\mathbb{T}^{3}$ case, since its proof highly relies on a key collapsing estimate, which fails in the $H^{1}$ energy space for the $\mathbb{T}^{3}$ case as proven in \cite{GSS14}. Thus, the $\mathbb{T}^{3}$ case requires new ideas. The torus case is the beginning to understand other related and important problems, such as the microscopic descriptions of the Mach number and Knudsen number and their limit to incompressible fluids.
\end{enumerate}

In this paper, our goal is to settle the above open problems.

\subsection{Statement of the Main Theorem}
Starting from the quantum $N$-body dynamics \eqref{equ:N-body schroedinger equation}, we take the normalization that
$\n{\psi_{N,\hbar}(t)}_{L_{X_{N}}^{2}}=1$,
and define the quantum mass density and momentum density by
\begin{align}\label{equ:quantum mass and momentum density}
&\rho_{N,\hbar}^{(k)}(t,X_{k})=\ga_{N,\hbar}^{(k)}(t,X_{k};X_{k}),\quad
J_{N,h}^{(1)}(t,x)=\operatorname{Im}\lrs{\hbar\nabla_{x_{1}} \ga_{N,\hbar}^{(1)}}(t,x;x).
\end{align}

The limiting macroscopic equation would be the compressible Euler-Poisson equation with a pressure term $P=\frac{b_{0}}{2}\rho^{2}$, which is
(in velocity form)
\begin{equation}\label{equ:euler equation}
\begin{cases}
& \partial _{t}\rho +\nabla \cdot \left( \rho u\right) =0, \\
& \partial _{t}u+(u\cdot \nabla )u+b_{0}\nabla_{x}\rho+\kappa \nabla_{x}(V_{c}*\rho) =0,\\
& (\rho,u)|_{t=0}=(\rho^{in},u^{in}),
\end{cases}
\end{equation}
or (in momentum form)
\begin{equation}\label{equ:euler equation,momentum}
\begin{cases}
& \partial _{t}\rho +\operatorname{div} J =0,\\
& \partial _{t}J+\operatorname{div}\lrs{\frac{J\otimes J}{\rho}}+\frac{1}{2}\nabla \lrs{b_{0}\rho^{2}}+ \kappa \rho\nabla_{x}(V_{c}*\rho)=0,\\
& (\rho,J)|_{t=0}=(\rho^{in},J^{in}).
\end{cases}
\end{equation}
Here, as usual,
\begin{align*}
&\rho(t,x):\R\times \R^{3}\mapsto \R\\
&u(t,x)=(u^{1}(t,x),u^{2}(t,x),u^{3}(t,x)):\R\times \R^{3}\mapsto \R^{3}\\
&J(t,x)=\lrs{\rho u}(t,x):\R\times \R^{3}\mapsto \R^{3}
\end{align*}
are respectively the mass density, the velocity, and the momentum of the fluid. The coupling constant $b_{0}=\int V$ is the macroscopic effect of the microscopic interaction $V$. When the coefficient $\kappa=0$, the system $(\ref{equ:euler equation})$ is reduced to a compressible Euler equation.
Specifically, we consider the initial data satisfying the condition
\begin{equation}\label{equ:initial condition,euler}
	\left\{
	\begin{aligned}
\rho^{in}\in H^{s-1}(\R^{3}),\quad u^{in}\in H^{s}(\R^{3}),\\
\rho^{in}(x)\geq 0,\quad \int_{\R^{3}}\rho^{in}(x)dx=1,
	\end{aligned}
	\right.
\end{equation}
with $s>\frac{9}{2}$ and $s\in \mathbb{N}$, so that the Euler-Poisson system $(\ref{equ:euler equation})$ has a unique solution $(\rho,u)$ up to some time $T_{0}$ such that\footnote{We are not dealing with sharp well-posedness of \eqref{equ:euler equation} here. The local well-posedness of the Euler system here is known by the standard theory on hyperbolic systems,see \cite{Maj84,Mak86,MUK86}.}
\begin{equation}\label{equ:euler-poisson equation, regularity condition}
	\left\{
	\begin{aligned}
&\rho\in C([0,T_{0}];H^{s-1}(\R^{3})),\quad u\in C([0,T_{0}];H^{s}(\R^{3})),\\
&\rho(t,x)\geq 0,\quad \int_{\R^{3}} \rho(t,x)dx=1.
	\end{aligned}
	\right.
\end{equation}
\begin{theorem}\label{thm:main theorem}
Let $\be\in(0,1)$, $\ka\geq 0$ and the marginal densities $\Ga_{N,\hbar}=\lr{\ga_{N,\hbar}^{(k)}}$ associated with $\psi_{N,\hbar}$ be the solution to the $N$-body dynamics \eqref{equ:N-body schroedinger equation} with a smooth compactly supported, spherically symmetric nonnegative potential $V$ and a repulsive Coulomb potential $V_{c}$. Assume the initial data satisfy the following conditions:

$(a)$ $\psi_{N,\hbar}(0)$ is normalized and the $N$-body energy bound holds:
\begin{align} \label{equ:n-body energy bound,initial data condition}
\lra{\psi_{N,\hbar}(0),(H_{N,\hbar}/N+1)^{2}\psi_{N,\hbar}(0)}\leq (E_{0})^{2},
\end{align}
for some $E_0>0$.

$(b)$ The initial data $(\rho^{in},u^{in})$ to $(\ref{equ:euler equation})$ satisfy condition $(\ref{equ:initial condition,euler})$ with $s=5$, and
the modulated energy between \eqref{equ:N-body schroedinger equation} and \eqref{equ:euler equation} at initial time tends to zero, that is, $\mathcal{M}(0)\to 0$ where
\begin{align*}
&\mathcal{M}(0)\\
:=&\int_{\mathbb{R}^{3N}}\vert\left(i \hbar \nabla_{x_{1}}-u(t,x_{1})\right)\psi_{N,\hbar}(0,X_N)\vert^2dX_N\\
&+\frac{N-1}{N}\int V_{N}(x_{1}-x_{2})\rho_{N,\hbar}^{(2)}(0,x_{1},x_{2})dx_{1}dx_{2}\\
&+b_{0}\int_{\mathbb{R}^{3}}\rho^{in}(x_{1})\rho^{in}(x_{1})dx_{1}-2b_{0}
\int_{\mathbb{R}^3}\rho^{in}(x_{1})\rho_{N,\hbar}^{(1)}(0,x_{1})dx_{1}\\
&+\int V_{c}(x_{1}-x_{2})\left[\frac{N-1}{N}\rho_{N,\hbar}^{(2)}(0,x_{1},x_{2}) +\rho^{in}(x_{1})\rho^{in}(x_{2})-
2\rho^{in}(x_{1})\rho_{N,\hbar}^{(1)}(0,x_{2})\right]dx_{1}dx_{2}.
\end{align*}

Then under the polynomial restriction\footnote{\eqref{equ:restriction,N,h} is in fact a rational restriction. We say polynomial to avoid confusing ``rational" and ``reasonable".}
\begin{align}\label{equ:restriction,N,h}
r(N,\hbar)=C(N^{\be-1}\hbar^{-6}+N^{-\frac{\be}{3}}\hbar^{-4}+N^{-\frac{1}{10}}\hbar^{-4})\to 0,
\end{align}
we have the quantitative estimates on the strong convergence of the mass density
\begin{align}\label{equ:the convergence of the mass denstiy}
& \n{\rho_{N,\hbar}^{(1)}(t,x)-\rho(t,x)}_{L_{t}^{\infty}[0,T_{0}]L_{x}^{2}(\R^{3})}^{2}\lesssim \mathcal{M}(0)+r(N,\hbar)+\hbar^{2},
\end{align}
and on the convergence of the momentum density
\begin{align}\label{equ;the convergence of the momentum density}
& \bbn{J_{N,\hbar}^{(1)}(t,x)-(\rho u)(t,x)}_{L_{t}^{\infty}[0,T_{0}]L_{x}^{1}(\R^{3})}^{2}\lesssim  \mathcal{M}(0)+r(N,\hbar)+\hbar^{2}.
\end{align}
\end{theorem}

When $\kappa>0$, Theorem \ref{thm:main theorem} is the first result which simultaneously deals with the $\delta$-type and Coulomb potentials and establishes the quantitative strong convergence to the full Euler-Poisson equation with pressure. Compared to \cite{GP21,CSZ23},
the emergence of the pressure term is the main novelty. We point out that, it is not clear if the scheme in \cite{CSWZ23,CSZ23} can handle the $\delta$-type and Coulomb potentials simultaneously,
since the energy estimates and collapsing estimates are totally different in the $\delta$-type and Coulomb potentials. Therefore, it requires completely new ideas for a simultaneous consideration of $\delta$-type and Coulomb potentials.

When $\kappa=0$, it reduces to the sole $\delta$-type potential case.
Compared with our previous work \cite{CSWZ23}, we here list the breakthroughs.
\begin{enumerate}
\item The parameter $\be$ is extended to the full range of $(0,1)$, which is almost optimal in the dilute regime.
\item The previous double-exponential restriction between $N$ and $\hbar$ is relaxed to be polynomial, which is a tremendous improvement.
\item Our new approach also works for the $\mathbb{T}^{3}$ case with slight modifications, as the proof is independent of the hardcore harmonic analysis on $\mathbb{T}^{3}$.
\end{enumerate}

Additionally, the convergence rate $\hbar^{2}$ should be optimal since the convergence rate of the modulated kinetic energy part at initial time is at most the order of $\hbar^{2}$. Besides, this can be achieved with WKB type initial data.

\subsection{Outline of the Proof}\label{section:Outline of the Proof}
The proof is based on a modulated energy method.\footnote{A closely related method is the relative entropy method, see for example, \cite{Yau91}.}
The modulated energy we use includes three parts
\begin{align}
\mathcal{M}(t)=\mathcal{M}_{K}(t)+\mathcal{F}_{\delta}(t)+\mathcal{F}_{c}(t),
\end{align}
where the kinetic energy part is
\begin{align}
\mathcal{M}_{K}(t)=
\int_{\mathbb{R}^{3N}}\vert\left(i \hbar \nabla_{x_{1}}-u(t,x_{1})\right)\psi_{N,\hbar}(t,X_N)\vert^2dX_N,
\end{align}
the $\delta$-type potential part is
\begin{align}
\mathcal{F}_{\delta}(t)=&\frac{N-1}{N}\iint_{\mathbb{R}^3\times\mathbb{R}^3}V_{N}(x_{1}-x_{2})\rho_{N,\hbar}^{(2)}(t,x_{1},x_{2})dx_{1}dx_{2}\\
&+b_{0}\int_{\mathbb{R}^{3}}\rho(t,x_{1})\rho(t,x_{1})dx_{1}-2b_{0}
\int_{\mathbb{R}^3}\rho(t,x_{1})\rho_{N,\hbar}^{(1)}(t,x_{1})dx_{1},\notag
\end{align}
and the Coulomb potential part is
\begin{align}
\mathcal{F}_{c}(t)=&
\iint_{\R^{3}\times\R^{3}} V_{c}(x_{1}-x_{2})\left[\frac{N-1}{N}\rho_{N,\hbar}^{(2)}(t,x_{1},x_{2})\right.\\
&\left.\quad \quad \quad +\rho(t,x_{1})\rho(t,x_{2})-
2\rho(t,x_{1})\rho_{N,\hbar}^{(1)}(t,x_{2})\right]dx_{1}dx_{2}.\notag
\end{align}

In Section \ref{section:The Time Evolution of the Modulated Energy}, we first derive the time evolution of the modulated energy
\begin{align}
\frac{d}{dt}\mathcal{M}(t)=\wt{\mathcal{M}}_{K}(t)+\wt{\mathcal{F}}_{\delta}(t)+\wt{\mathcal{F}}_{c}(t),
\end{align}
where the kinetic energy contribution part is
\begin{align*}
\wt{\mathcal{M}}_{K}(t)=&
-\sum_{j,k=1}^{3}\int_{\mathbb{R}^{3N}}\lrs{\pa_{j}u^{k}+\pa_{k}u^{j}}(-i\hbar\partial_j\psi_{N,\hbar}-u^{j}\psi_{N,\hbar})
\overline{(-i\hbar\partial_k\psi_{N,\hbar}-u^{k}\psi_{N,\hbar})}dX_N\notag\\
 &+\frac{\hbar^2}{2}\int_{\mathbb{R}^{3}}\Delta(\operatorname{div} u)(t,x_{1})\rho_{N,\hbar}^{(1)}(t,x_{1})dx_{1},
\end{align*}
with the notations $u=(u^{1},u^{2},u^{3})$, $x_{1}=(x_{1}^1,x_{1}^2,x_{1}^3)\in\mathbb{R}^3$ and $\partial_j=\partial_{x_{1}^j}$,
the $\delta$-type potential contribution part is
 \begin{align}
\wt{\mathcal{F}}_{\delta}(t)=&
\frac{N-1}{N}\int (u(t,x_{1})-u(t,x_{2}))\nabla V_N(x_{1}-x_{2})\rho_{N,\hbar}^{(2)}(t,x_{1},x_{2})dx_{1}dx_{2}\\
&-b_{0}\int \operatorname{div} u(t,x_{1})\rho(t,x_{1})\lrc{\rho(t,x_{1})-2\rho_{N,\hbar}^{(1)}(t,x_{1})}dx_{1},\notag
\end{align}
and the Coulomb potential contribution part is
\begin{align}
\wt{\mathcal{F}}_{c}(t)=&
\int(u(t,x_{1})-u(t,x_{2}))\nabla V_c(x_{1}-x_{2})\\
&\left[\frac{N-1}{N}\rho_{N,\hbar}^{(2)}(t,x_{1},x_{2})+\rho(t,x_{1})\rho(t,x_{2})-
2\rho(t,x_{1})\rho_{N,\hbar}^{(1)}(t,x_{2})\right]dx_{1}dx_{2}.\notag
\end{align}

It is easy to control the kinetic energy contribution part
\begin{align}
\wt{\mathcal{M}}_{K}(t)\lesssim \mathcal{M}_{K}(t)+\hbar^{2}.
\end{align}
The toughest part in the modulated energy method is to control the potential contribution part both in the classical and quantum setting.
See, for example, \cite{Due16,GP21,LZ06,Ser17,Ser20,Zha02}.
In the classical mean-field limit with Coulomb potential, Serfaty in \cite[Proposition 1.1]{Ser20} establishes a crucial functional inequality to solve this challenging problem. Then for the quantum many-body systems with Coulomb potential, based on Serfaty's inequality,
 Golse and Paul in \cite{GP21} managed to control the Coulomb potential contribution part $\wt{\mathcal{F}}_{c}(t)$ as follows
\begin{align}
\wt{\mathcal{F}}_{c}(t)\lesssim & \mathcal{F}_{c}(t)+CN^{-\frac{1}{3}},\label{equ:functional inequality,fc}\\
0\leq& \mathcal{F}_{c}(t)+CN^{-\frac{2}{3}}.\label{equ:lower bound,fc}
\end{align}
Serfaty's inequality is a special and impressive tool based on deep observations of the structure of Coulomb potential.
It is limited to a special class of singular potentials, as its proof highly relies on the structure and the profile of the potentials, such as the Coulomb characteristic that $-\Delta V_{c}=c_{0}\delta$. Therefore, it is quite difficult to establish a Serfaty's inequality for the $\delta$-type potential case, because of the general profile and sharp singularity of the $\delta$-type potential. In fact, due to the presence of the singular correlation structure caused by the $\delta$-type potential, the analysis would be totally different and is expected to be rather intricate.\smallskip

In this paper, we develop a new scheme without using Serfaty's inequality to control the $\delta$-type potential parts $\mathcal{F}_{\delta}(t)$ and $\wt{\mathcal{F}}_{\delta}(t)$ and establish
\begin{align}
\wt{\mathcal{F}}_{\delta}(t) \lesssim&  \mathcal{F}_{\delta}(t)
+r(N,\hbar),\label{equ:functional inequality,fdelta,outline}\\
0 \leq& \mathcal{F}_{\delta}(t)+r(N,\hbar).\label{equ:lower bound,fdelta,outline}
\end{align}
The proof is divided in several steps.\smallskip

\textbf{Step 1. Preliminary reduction.}
Applying the approximation of identity to the one-body term of $\mathcal{F}_{\delta}(t)$, we have the approximation
\begin{align}
\mathcal{F}_{\delta}(t)\sim &\int V_{N}(x-y) \lrc{\frac{N-1}{N}\rho_{N,\hbar}^{(2)}(t,x,y)\right.\\
&\quad \quad \left.-\rho_{N,\hbar}^{(1)}(t,x)\rho(t,y)-\rho(t,x)\rho_{N,\hbar}^{(1)}(t,y)+\rho(t,x)\rho(t,y)}dxdy.\notag
\end{align}
To have a closed estimate, namely, letting $\wt{\mathcal{F}}_{\delta}(t)$ match the approximation of $\mathcal{F}_{\delta}(t)$, we get by integration by parts for the two-body term that
 \begin{align*}
\wt{\mathcal{F}}_{\delta}(t)=&
\frac{N-1}{N}\int \operatorname{div} u(t,x_{1}) V_N(x_{1}-x_{2})\rho_{N,\hbar}^{(2)}(t,x_{1},x_{2})dx_{1}dx_{2}\\
&-\int  (u(t,x)-u(t,y)) V_{N}(x-y)\nabla_{x}\rho_{N,\hbar}^{(2)}(t,x,y)dxdy\\
&-b_{0}\int \operatorname{div} u(t,x_{1})\rho(t,x_{1})\lrc{\rho(t,x_{1})-2\rho_{N,\hbar}^{(1)}(t,x_{1})}dx_{1}.\notag
\end{align*}
Using the approximation of identity to the one-body term again, we decompose $\wt{\mathcal{F}}_{\delta}(t)$ into the main part and error part
\begin{align}
\wt{\mathcal{F}}_{\delta}(t)=MP+EP,
\end{align}
where
\begin{align}
MP \sim &-\int \operatorname{div} u(t,x) V_{N}(x-y)
 \lrc{\frac{N-1}{N}\rho_{N,\hbar}^{(2)}(t,x,y)\right. \label{equ:main part,outline}\\
& \left. \quad \quad   -\rho_{N,\hbar}^{(1)}(t,x)\rho(t,y)-\rho(t,x)\rho_{N,\hbar}^{(1)}(t,y)+\rho(t,x)\rho(t,y)}dxdy,\notag\\
EP=&-\int  (u(t,x)-u(t,y)) V_{N}(x-y)\nabla_{x}\rho_{N,\hbar}^{(2)}(t,x,y)dxdy.\label{equ:error part,outline}
\end{align}
Such a decomposition is based on the key observation that
the difference coupled with the $\delta$-type potential
\begin{align}\label{equ:cancellation structure,delta potential,outline}
(u(t,x)-u(t,y))V_{N}(x-y),
\end{align}
 when it is tested against a regular function, would vanish in the $N\to\infty$ limit. Such a structure is notably special for the $\delta$-type potential, since the difference coupled with a common potential including the Coulomb case cannot provide any smallness.

To prove that the error part \eqref{equ:error part,outline} is indeed a small term, it requires the regularity of the two-body density function. Therefore, we delve into the analysis of two-body energy estimates, then deal with the error part and the main part in the Step 3 and 4 respectively.

\textbf{Step 2. Two-body energy estimate.} As usual, a-priori estimates are
one of the toughest parts in the study of many-body dynamics as one must
seek a regularity high enough for the limiting argument and at the same time
low enough that it is provable. In Section
\ref{section:Energy Estimate Using Singular Correlation Structure}, we prove that the wave function with
added the singular correlation structure satisfies the two-body $H^{1}$
energy bound
\begin{equation}
\Big \langle (1-\hbar ^{2}\Delta _{x_{1}})(1-\hbar ^{2}\Delta _{x_{2}})\frac{\psi
_{N,\hbar }(t,X_{N})}{1-w_{N,\hbar }(x_{1}-x_{2})},\frac{\psi _{N,\hbar
}(t,X_{N})}{1-w_{N,\hbar }(x_{1}-x_{2})}\Big \rangle \leq C,
\label{eqn:2-particle H^1 estimate}
\end{equation}%
where $w_{N,\hbar }(x)$ satisfies the zero-energy scattering equation
\begin{equation}
\left\{ \begin{aligned} &\left(
-\hbar^{2}\Delta+\frac{1}{N}V_{N}(x)\right)(1-w_{N,\hbar}(x))=0,\\
&\lim_{|x|\to \infty} w_{N,\hbar}(x)=0. \end{aligned}\right.
\end{equation}%
The singular correlation function $w_{N,\hbar }(x)$ varies effectively on
the short scale for $|x|\lesssim N^{-\beta }$ and has the same singularity
as the Coulomb potential at infinity.

One of the main difficulties here is to understand the interparticle
singular correlation structure generated by the $\delta$-type potential. See,
for example, \cite{LSSY05} for the study of the static case of Bose gas. For
the time-dependent systems, Erd{\"{o}}s, Schlein, and Yau \cite%
{ESY06,ESY09,ESY10} first introduced the two-body energy estimate which
plays a central role in the derivation of Gross-Pitaevskii equation with the
nonlinear interaction given by a scattering length. However, instead of
showing the emergence of the scattering length, our purpose here is proving
the functional inequalities \eqref{equ:functional inequality,fdelta,outline}
and \eqref{equ:lower bound,fdelta,outline}.

Another difficulty lies in the Coulomb singularity. The Coulomb potential,
if taken to high powers, results in singularities which cannot be controlled
by derivatives. The $(H_{N,\hbar })^{2}$ energy estimate (\ref%
{eqn:2-particle H^1 estimate}) we prove (and require here) is at the
borderline case. Indeed, the square of the Coulomb potential is bounded with
respect to the kinetic energy in the sense that as operators
$|V_{c}(x)|^{2}\leq C(1-\Delta _{x})$.
However, no such estimates hold for $|V_{c}(x)|^{3}$ due to the singularity
of the origin.

\textbf{Step 3. Analysis of the Error Part.}
After setting up the energy estimates, we begin to analyze the error part \eqref{equ:error part,outline}. Because of the presence of the singular correlation structure,
the two-body density function lacks the a-priori energy bound but can be decomposed into the singular and regular (relatively speaking) parts
\begin{align}\label{equ:decomposition,singular,regular,outline}
\rho_{N,\hbar}^{(2)}(t,x,y)=(1-w_{N,\hbar}(x-y))^{2} \frac{\rho_{N,\hbar}^{(2)}(t,x,y)}{(1-w_{N,\hbar}(x-y))^{2}}.
\end{align}
Hence, we need to rewrite the error part \eqref{equ:error part,outline} as
\begin{align*}
&\int (u(t,x)-u(t,y))\cdot  V_{N}(x-y) \nabla_{x}\lrc{(1-w_{N,\hbar}(x-y))^{2} \frac{\rho_{N,\hbar}^{(2)}(t,x,y)}{(1-w_{N,\hbar}(x-y))^{2}}}dxdy\\
=&\int (u(t,x)-u(t,y))\cdot  V_{N}(x-y) \lrs{\nabla_{x}(1-w_{N,\hbar}(x-y))^{2}} \frac{\rho_{N,\hbar}^{(2)}(t,x,y)}{(1-w_{N,\hbar}(x-y))^{2}}dxdy\\
&+\int (u(t,x)-u(t,y))\cdot  V_{N}(x-y) (1-w_{N,\hbar}(x-y))^{2}\nabla_{x}\lrc{ \frac{\rho_{N,\hbar}^{(2)}(t,x,y)}{(1-w_{N,\hbar}(x-y))^{2}}}dxdy.
\end{align*}
When the derivative hits the singular correlation function, it produces singularities by the defining feature of the singular correlation function, which would give a rise of $O(N^{\be})$.
On the other hand, when the derivative hits the (relatively) regular part, it still requires a careful analysis as we have limited regularity as discussed before on the modified two-body density function.

In Section \ref{section:Error Analysis of Two-Body Term}, we prove that, the cancellation structure \eqref{equ:cancellation structure,delta potential,outline} indeed dominates the singularity generated by the $\delta$-type potential and singular correlation function, and obtain the error estimate
\begin{align}
EP\lesssim N^{\be-1}\hbar^{-6}+N^{-\frac{\be}{2}}\hbar^{-4}.
\end{align}

\textbf{Step 4. Analysis of the Main Part.}
One difficulty of the analysis of the main part \eqref{equ:main part,outline} is the sharp singularity and the unknown profile of $V_{N}(x)$.
To overcome it, our strategy is to replace $V_{N}(x)$ with a slowly varying potential $G_{N}(x)$
which enjoys a number of good properties, but it comes at a price of the integrand's regularity.
Thus, for the main part \eqref{equ:main part,outline}, we again need to decompose the two-body density function into the singular part and relatively regular part as follows
\begin{align}
MP=&\int \operatorname{div} u(t,x) V_{N}(x-y)
 \lrc{\frac{N-1}{N}\frac{\rho_{N,\hbar}^{(2)}(t,x,y)}{(1-w_{N,\hbar}(x-y))^{2}}(1-w_{N,\hbar}(x-y))^{2} \right. \\
& \left. \quad \quad   -\rho_{N,\hbar}^{(1)}(t,x)\rho(t,y)-\rho(t,x)\rho_{N,\hbar}^{(1)}(t,y)+\rho(t,x)\rho(t,y)}dxdy.\notag
\end{align}
Note that $(1-w_{N,\hbar}(x-y))^{2}\sim 1+O(w_{N,\hbar}(x-y))$. Then
by the two-body energy bound and the property for the scattering function $w_{N,\hbar}(x-y)$, we can prove that
\begin{align}
MP\sim &-\int \operatorname{div} u(t,x) V_{N}(x-y)
 \lrc{\frac{N-1}{N}\frac{\rho_{N,\hbar}^{(2)}(t,x,y)}{(1-w_{N,\hbar}(x-y))^{2}} \right. \\
& \left. \quad \quad   -\rho_{N,\hbar}^{(1)}(t,x)\rho(t,y)-\rho(t,x)\rho_{N,\hbar}^{(1)}(t,y)+\rho(t,x)\rho(t,y)}dxdy.\notag
\end{align}
Since the integrand now enjoys the energy bound, we are able to replace $V_{N}$ by $G_{N}$ and get
\begin{align}
MP\sim &-b_{0}\int \operatorname{div} u(t,x) G_{N}(x-y)
 \lrc{\frac{N-1}{N}\frac{\rho_{N,\hbar}^{(2)}(t,x,y)}{(1-w_{N,\hbar}(x-y))^{2}} \right. \\
& \left. \quad \quad   -\rho_{N,\hbar}^{(1)}(t,x)\rho(t,y)-\rho(t,x)\rho_{N,\hbar}^{(1)}(t,y)+\rho(t,x)\rho(t,y)}dxdy,\notag
\end{align}
where $G_{N}(x)=N^{3\eta}G(N^{\eta}x)$ with $\eta<\frac{1}{3}$.

In Section \ref{section:Tamed Singularities}, we will give a detailed proof of the above analysis and
and arrive at the approximations of $\mathcal{F}_{\delta}(t)$ and $\wt{\mathcal{F}}_{\delta}(t)$ given by
\begin{align}
\mathcal{F}_{\delta}(t)\sim &b_{0}\int G_{N}(x-y) \lrc{\frac{N-1}{N}\rho_{N,\hbar}^{(2)}(t,x,y)\right.\label{equ:approximation,Fdelta}\\
&\quad \quad \left.-\rho_{N,\hbar}^{(1)}(t,x)\rho(t,y)-\rho(t,x)\rho_{N,\hbar}^{(1)}(t,y)+\rho(t,x)\rho(t,y)}dxdy,\notag
\end{align}
and
\begin{align}
\wt{\mathcal{F}}_{\delta}(t)\sim &-b_{0}\int \operatorname{div} u(t,x) G_{N}(x-y)
 \lrc{\frac{N-1}{N}\rho_{N,\hbar}^{(2)}(t,x,y)\right.\label{equ:approximation,evolution,Fdelta}\\
& \left. \quad \quad   -\rho_{N,\hbar}^{(1)}(t,x)\rho(t,y)-\rho(t,x)\rho_{N,\hbar}^{(1)}(t,y)+\rho(t,x)\rho(t,y)}dxdy.\notag
\end{align}

Now, from the approximations of $\mathcal{F}_{\delta}(t)$ and $\wt{\mathcal{F}}_{\delta}(t)$, we are left to prove a reduced form of the functional inequality
\begin{align*}
 &\int \operatorname{div} u(x) G_{N}(x-y)
 \lrc{\frac{N-1}{N}\rho_{N,\hbar}^{(2)}(x,y)-\rho_{N,\hbar}^{(1)}(x)\rho(y)-\rho(x)\rho_{N,\hbar}^{(1)}(y)+\rho(x)\rho(y)}dxdy\\
 \lesssim& \int G_{N}(x-y) \lrc{\frac{N-1}{N}\rho_{N,\hbar}^{(2)}(x,y)-\rho_{N,\hbar}^{(1)}(x)\rho(y)-\rho(x)\rho_{N,\hbar}^{(1)}(y)+\rho(x)\rho(y)}dxdy +o(1),
\end{align*}
  which looks more concise and tractable than the original functional inequality \eqref{equ:functional inequality,fdelta,outline}. But, it is unknown if the integrand
\begin{align}\label{equ:two-body structure}
\frac{N-1}{N}\rho_{N,\hbar}^{(2)}(x,y)-\rho_{N,\hbar}^{(1)}(x)\rho(y)-\rho(x)\rho_{N,\hbar}^{(1)}(y)+\rho(x)\rho(y)
\end{align}
is non-negative. We cannot simply rule out the term $\operatorname{div} u(x)$ either. Thus, it is still non-trivial to deduce the inequality.
In fact, as we will see in Section \ref{section:Reduced Version of Functional Inequality}, the special structure \eqref{equ:two-body structure} with a slowly varying potential $G_{N}(x)$ plays a crucial role in establishing the reduced version of functional inequality. Then, at the end of Section \ref{section:Reduced Version of Functional Inequality}, we conclude the functional inequalities \eqref{equ:functional inequality,fdelta,outline} and \eqref{equ:lower bound,fdelta,outline}.

Finally in Section \ref{section:Concluding the Strong Convergence of Quantum Densities}, by using functional inequalities on $\wt{\mathcal{F}}_{\delta}(t)$ and $\wt{\mathcal{F}}_{c}(t)$,
 we prove the Gronwall's inequality for the positive modulated energy
 \begin{align*}
 &\frac{d}{dt}\mathcal{M}^{+}(t)
\lesssim \mathcal{M}^{+}(t)+
\hbar^{2},
 \end{align*}
 where $\mathcal{M}^{+}(t)=\mathcal{M}(t)+2r(N,\hbar)$.
Subsequently, with the quantitative convergence rate of the positive modulated energy, we further conclude the quantitative strong convergence rate of quantum mass and momentum densities, in which the $\delta$-type potential part plays an indispensable role in upgrading to the quantitative strong convergence.

\section{The Time Evolution of the Modulated Energy} \label{section:The Time Evolution of the Modulated Energy}
We consider the modulated
energy in the quantum $N$-body dynamics corresponding to the $\delta$-type and Coulomb potentials
\begin{align}\label{equ:modulated energy}
\mathcal{M}(t):=&\int_{\mathbb{R}^{3N}}\vert\left(i \hbar \nabla_{x_{1}}-u(t,x_{1})\right)\psi_{N,\hbar}(t,X_N)\vert^2dX_N+\mathcal{F}_{\delta}(t)+\mathcal{F}_{c}(t),\notag
\end{align}
where the $\delta$-type potential part is
\begin{align}
\mathcal{F}_{\delta}(t)=&\frac{N-1}{N}\iint_{\mathbb{R}^3\times\mathbb{R}^3}V_{N}(x_{1}-x_{2})\rho_{N,\hbar}^{(2)}(t,x_{1},x_{2})dx_{1}dx_{2}\\
&+b_{0}\int_{\mathbb{R}^{3}}\rho(t,x_{1})\rho(t,x_{1})dx_{1}-2b_{0}
\int_{\mathbb{R}^3}\rho(t,x_{1})\rho_{N,\hbar}^{(1)}(t,x_{1})dx_{1},\notag
\end{align}
and the Coulomb potential part is
\begin{align}
\mathcal{F}_{c}(t)=&
\iint_{\R^{3}\times\R^{3}} V_{c}(x_{1}-x_{2})\left[\frac{N-1}{N}\rho_{N,\hbar}^{(2)}(t,x_{1},x_{2})\right.\\
&\left.\quad \quad \quad +\rho(t,x_{1})\rho(t,x_{2})-
2\rho(t,x_{1})\rho_{N,\hbar}^{(1)}(t,x_{2})\right]dx_{1}dx_{2}.\notag
\end{align}
Here, we might as well assume that the coefficient $\kappa=1$, as the proof works the same for $\kappa\geq 0$.

First,
we need to derive a time evolution equation for $\mathcal{M}(t)$.
The related quantities for $\psi_{N,\hbar}$ are given as the following.
\begin{lemma}\label{lemma:conservation law of mass}
We have the following computations regarding $\psi_{N,\hbar}$:
\begin{align}
&\pa_{t}\rho_{N,\hbar}^{(1)}+\operatorname{div}J_{N,\hbar}^{(1)}=0,\label{equ:mass conservation}\\
&\pa_{t}J_{N,\hbar}^{(1)}=\frac{\hbar^{2}}{2}\int \operatorname{Re}\lrs{(-\Delta_{x_{1}}\overline{\psi_{N,\hbar}})  \nabla_{x_{1}}\psi_{N,\hbar}
+\overline{\psi_{N,\hbar}}  \nabla_{x_{1}}\Delta_{x_{1}}\psi_{N,\hbar}} dX_{2,N}\label{equ:momontum conservation}\\
&\quad \quad \quad \quad -\frac{N-1}{N}\int \nabla_{x_{1}}(V_{N}+V_{c})(x_{1}-x_{2})\rho_{N,\hbar}^{(2)}(t,x_{1},x_{2})dx_{2},\notag
\\
&E_{N,\hbar}(t)= E_{N,\hbar}(0)\leq E_{0},\label{equ:energy conservation law}
\end{align}
where $X_{2,N}=(x_{2},...,x_{N})$ and the momentum density $J_{N,\hbar}^{(1)}(t,x_{1})$ and the energy $E_{N,\hbar}(t)$ are defined by
\begin{align}
J_{N,\hbar}^{(1)}(t,x_{1})=&\operatorname{Im}\lrs{\hbar\nabla_{x_{1}} \ga_{N,\hbar}^{(1)}}(t,x_{1};x_{1})
=\hbar\int \operatorname{Im}(\overline{\psi_{N,\hbar}}  \nabla_{x_{1}}\psi_{N,\hbar})(t,X_N)dX_{2,N},\label{equ:momentum density}\\
E_{N,\hbar}(t)=&\frac{1}{N}\lra{(H_{N,\hbar}+N)\psi_{N,\hbar}(t),\psi_{N,\hbar}(t)}.
\end{align}

\end{lemma}
\begin{proof}
As the mass and energy conservation laws are well-known, we omit the proof of \eqref{equ:mass conservation} and \eqref{equ:energy conservation law}. We provide the proof of the evolution \eqref{equ:momontum conservation} of the momentum density. From \eqref{equ:momentum density},
 we can write out
\begin{align*}
\pa_{t}J_{N,\hbar}^{(1)}=&\hbar \int \operatorname{Im}\lrs{\overline{\pa_{t}\psi_{N,\hbar}}  \nabla_{x_{1}}\psi_{N,\hbar}
+\overline{\psi_{N,\hbar}}  \nabla_{x_{1}}\pa_{t}\psi_{N,\hbar}} dX_{2,N}\\
=&\int \operatorname{Im}\lrs{i\overline{H_{N,\hbar}\psi_{N,\hbar}}  \nabla_{x_{1}}\psi_{N,\hbar}
-i\overline{\psi_{N,\hbar}}  \nabla_{x_{1}}H_{N,\hbar}\psi_{N,\hbar}}dX_{2,N}\\
=&\int \operatorname{Re}\lrs{\overline{H_{N,\hbar}\psi_{N,\hbar}}  \nabla_{x_{1}}\psi_{N,\hbar}
-\overline{\psi_{N,\hbar}}  \nabla_{x_{1}}H_{N,\hbar}\psi_{N,\hbar}} dX_{2,N}\\
=&I_{K}+I_{V},
\end{align*}
where
\begin{align*}
I_{K}=&\frac{\hbar^{2}}{2}\int \operatorname{Re}\lrs{\sum_{i=1}^{N}(-\Delta_{x_{i}}\overline{\psi_{N,\hbar}})  \nabla_{x_{1}}\psi_{N,\hbar}
-\overline{\psi_{N,\hbar}}  \nabla_{x_{1}}\sum_{i=1}^{N}-\Delta_{x_{i}}\psi_{N,\hbar}} dX_{2,N},
\end{align*}
and
\begin{align*}
I_{V}=&\int \operatorname{Re}\lrs{\frac{1}{N}\sum_{i<j}^{N}(V_{N}+V_{c})(x_{i}-x_{j})\overline{\psi_{N,\hbar}}  \nabla_{x_{1}}\psi_{N,\hbar}}dX_{2,N}\\
&-\int \operatorname{Re}\lrs{\overline{\psi_{N,\hbar}}  \nabla_{x_{1}}\frac{1}{N}\sum_{i<j}^{N}(V_{N}+V_{c})(x_{i}-x_{j})\psi_{N,\hbar}} dX_{2,N}.\notag
\end{align*}

For $I_{K}$, we use integration by parts with $\Delta_{x_{i}}$ to obtain
\begin{align*}
I_{K}=\frac{\hbar^{2}}{2}\int \operatorname{Re}\lrs{(-\Delta_{x_{1}}\overline{\psi_{N,\hbar}})  \nabla_{x_{1}}\psi_{N,\hbar}
+\overline{\psi_{N,\hbar}}  \nabla_{x_{1}}\Delta_{x_{1}}\psi_{N,\hbar}} dX_{2,N},
\end{align*}
where the other $i$-summands vanish when $i\geq 2$.

For $I_{V}$, we note that the $i$-summands also vanish when $i\geq 2$ and hence have
\begin{align*}
I_{V}=&\int \operatorname{Re}\lrs{\frac{1}{N}\sum_{j=2}^{N}(V_{N}+V_{c})(x_{1}-x_{j})\overline{\psi_{N,\hbar}}  \nabla_{x_{1}}\psi_{N,\hbar}}dX_{2,N}\\
&-\int \operatorname{Re}\lrs{\overline{\psi_{N,\hbar}}  \nabla_{x_{1}}\frac{1}{N}\sum_{j=2}^{N}(V_{N}+V_{c})(x_{1}-x_{j})\psi_{N,\hbar}} dX_{2,N}\\
=&-\frac{N-1}{N}\int |\psi_{N,\hbar}|^{2}\nabla_{x_{1}}(V_{N}+V_{c})(x_{1}-x_{2})dX_{2,N}\\
=&-\frac{N-1}{N}\int \nabla_{x_{1}}(V_{N}+V_{c})(x_{1}-x_{2})\rho_{N,\hbar}^{(2)}(t,x_{1},x_{2})dx_{2}.
\end{align*}
This completes the proof of \eqref{equ:momontum conservation}.
\end{proof}

Now, we derive the time evolution of $\mathcal{M}(t)$.
\begin{proposition}\label{prop:the evolution of the modulated energy}
Let $\mathcal{M}(t)$ be defined in $\eqref{equ:modulated energy}$, there holds
  \begin{align}\label{equ:evolution of modulated energy}
 &\frac{d}{dt}\mathcal{M}(t)\\
 =&-\sum_{j,k=1}^{3}\int_{\mathbb{R}^{3N}}\lrs{\pa_{j}u^{k}+\pa_{k}u^{j}}(-i\hbar\partial_j\psi_{N,\hbar}-u^{j}\psi_{N,\hbar})
\overline{(-i\hbar\partial_k\psi_{N,\hbar}-u^{k}\psi_{N,\hbar})}dX_N\notag\\
 &+\frac{\hbar^2}{2}\int_{\mathbb{R}^{3}}\Delta(\operatorname{div} u)(t,x_{1})\rho_{N,\hbar}^{(1)}(t,x_{1})dx_{1}+\wt{\mathcal{F}}_{\delta}(t)+
 \wt{\mathcal{F}}_{c}(t),\notag
 \end{align}
where we used the notations $u=(u^{1},u^{2},u^{3})$, $x_{1}=(x_{1}^1,x_{1}^2,x_{1}^3)\in\mathbb{R}^3$ and $\partial_j=\partial_{x_{1}^j}$. Here, the $\delta$-type potential contribution part is
 \begin{align}\label{equ:evolution,Fdelta}
\wt{\mathcal{F}}_{\delta}(t)=&
\frac{N-1}{N}\int (u(t,x_{1})-u(t,x_{2}))\nabla V_N(x_{1}-x_{2})\rho_{N,\hbar}^{(2)}(t,x_{1},x_{2})dx_{1}dx_{2}\\
&-b_{0}\int \operatorname{div} u(t,x_{1})\rho(t,x_{1})\lrc{\rho(t,x_{1})-2\rho_{N,\hbar}^{(1)}(t,x_{1})}dx_{1},\notag
\end{align}
and the Coulomb potential contribution part is
\begin{align}\label{equ:evolution,Fc}
\wt{\mathcal{F}}_{c}(t)=&
\int(u(t,x_{1})-u(t,x_{2}))\nabla V_c(x_{1}-x_{2})\\
&\left[\frac{N-1}{N}\rho_{N,\hbar}^{(2)}(t,x_{1},x_{2})+\rho(t,x_{1})\rho(t,x_{2})-
2\rho(t,x_{1})\rho_{N,\hbar}^{(1)}(t,x_{2})\right]dx_{1}dx_{2}.\notag
\end{align}

\end{proposition}
\begin{proof}
We decompose the modulated energy into five parts to do the calculation.

$$\begin{aligned}
\mathcal{M}_1(t)&=\int_{\mathbb{R}^{3N}}\vert i \hbar \nabla_{x_{1}}\psi_{N,\hbar}(t,X_N)\vert^2dX_N\\
                &+\frac{N-1}{N}\iint_{\mathbb{R}^3\times\mathbb{R}^3}(V_{N}+V_{c})(x_{1}-x_{2})\rho_{N,\hbar}^{(2)}(t,x_{1},x_{2})dx_{1}dx_{2},\\
\mathcal{M}_2(t)&=i\hbar \int_{\mathbb{R}^{3N}}u(t,x_{1})(\overline{\psi_{N,\hbar}}  \nabla_{x_{1}}\psi_{N,\hbar}
-\psi_{N,\hbar}\nabla_{x_{1}}\overline{\psi_{N,\hbar}})(t,X_N)dX_N,\\
\mathcal{M}_3(t)&=\int_{\mathbb{R}^{3N}}\vert u(t,x_{1})\psi_{N,\hbar}(t,X_N)\vert^2dX_N,\\
\mathcal{M}_4(t)&=b_{0}\int_{\mathbb{R}^{3}}\rho(t,x_{1})\rho(t,x_{1})dx_{1}+
\int_{\mathbb{R}^{3}}\rho(t,x_{1})(V_{c}*\rho)(t,x_{1})dx_{1},\\
\mathcal{M}_{5}(t)&=-2b_{0}\int_{\mathbb{R}^3}\rho(t,x_{1})\rho_{N,\hbar}^{(1)}(t,x_{1})dx_{1}-2
\int (V_{c}*\rho)(t,x_{1})\rho_{N,\hbar}^{(1)}(t,x_{1})dx_{1}.
\end{aligned}$$

For $\mathcal{M}_1(t)$, by the symmetry of the wave function $\psi_{N,\hbar}(t)$, we obtain
\begin{align*}
\mathcal{M}_1(t)
=&\int_{\mathbb{R}^{3N}}\left(-\hbar^2\Delta_{x_{1}}\psi_{N,\hbar}+\frac{N-1}{N}(V_{N}+V_{c})(x_{1}-x_{2})\psi_{N,\hbar}
\right)\overline{\psi_{N,\hbar}}dX_N\\
=&\frac{2}{N}\lra{H_{N,\hbar}\psi_{N,\hbar}(t),\psi_{N,\hbar}(t)}\\
=&\frac{2}{N}\lra{H_{N,\hbar}\psi_{N,\hbar}(0),\psi_{N,\hbar}(0)},
\end{align*}
where in the last equality we have used the conservation of energy. Therefore, we have that $$\frac{d}{dt}\mathcal{M}_1(t)=0.$$

For $\mathcal{M}_2(t)$, from the definition of $J_{N,\hbar}^{(1)}(t,x)$ in \eqref{equ:momentum density}, we note that
\begin{align*}
\mathcal{M}_{2}(t)=&i\hbar \int_{\mathbb{R}^{3N}}u(t,x_{1})(\overline{\psi_{N,\hbar}}  \nabla_{x_{1}}\psi_{N,\hbar}
-\psi_{N,\hbar}\nabla_{x_{1}}\overline{\psi_{N,\hbar}})(t,X_N)dX_N\\
=&-2\int u(t,x_{1}) \hbar\int \operatorname{Im}(\overline{\psi_{N,\hbar}}  \nabla_{x_{1}}\psi_{N,\hbar})(t,X_N)dX_{2,N}dx_{1}\\
=&-2\int u(t,x_{1})J_{N,\hbar}^{(1)}(t,x_{1})dx_{1}.
\end{align*}
Thus, we have
\begin{align}\label{equ:estimate,evolution,M2}
\frac{d}{dt}\mathcal{M}_2(t)=&-2\int \pa_{t}u(t,x_{1}) J_{N,\hbar}^{(1)}(t,x_{1})dx_{1}-2\int u(t,x_{1}) \pa_{t}J_{N,\hbar}^{(1)}(t,x_{1})dx_{1}.
\end{align}
For the second term on the r.h.s of \eqref{equ:estimate,evolution,M2}, by \eqref{equ:momontum conservation} we obtain
\begin{align}
&-2\int u(t,x_{1}) \pa_{t}J_{N,\hbar}^{(1)}(t,x_{1})dx_{1}\notag\\
=&-\hbar^{2}\int u(t,x_{1}) \operatorname{Re}\lrs{(-\Delta_{x_{1}}\overline{\psi_{N,\hbar}})  \nabla_{x_{1}}\psi_{N,\hbar}
+\overline{\psi_{N,\hbar}}  \nabla_{x_{1}}\Delta_{x_{1}}\psi_{N,\hbar}} dX_{N}\notag\\
&+\frac{2(N-1)}{N}\int u(t,x_{1})  \nabla_{x_{1}}(V_{N}+V_{c})(x_{1}-x_{2})\rho_{N,\hbar}^{(2)}(t,x_{1},x_{2})dx_{1}dx_{2}\notag\\
=&-\hbar^{2}\int u(t,x_{1}) \operatorname{Re}\lrs{(-\Delta_{x_{1}}\overline{\psi_{N,\hbar}})  \nabla_{x_{1}}\psi_{N,\hbar}
+\overline{\psi_{N,\hbar}}  \nabla_{x_{1}}\Delta_{x_{1}}\psi_{N,\hbar}} dX_{N}\label{equ:evolution,M2,momentum,kinetic}\\
&+\frac{(N-1)}{N}\int (u(t,x_{1})-u(t,x_{2}))  \nabla_{x_{1}}(V_{N}+V_{c})(x_{1}-x_{2})\rho_{N,\hbar}^{(2)}(t,x_{1},x_{2})dx_{1}dx_{2},\notag
\end{align}
where in the last equality we used the antisymmetry of $\nabla (V_{N}+V_{c})$.

Next, we deal with \eqref{equ:evolution,M2,momentum,kinetic}. By integration by parts, we obtain
\begin{align}
&-\hbar^{2}\int u(t,x_{1}) \operatorname{Re}\lrs{(-\Delta_{x_{1}}\overline{\psi_{N,\hbar}})  \nabla_{x_{1}}\psi_{N,\hbar}
+\overline{\psi_{N,\hbar}}  \nabla_{x_{1}}\Delta_{x_{1}}\psi_{N,\hbar}} dX_{N}\notag\\
=&-\hbar^{2}\operatorname{Re}\int 2u(t,x_{1})(-\Delta_{x_{1}}\overline{\psi_{N,\hbar}})  \nabla_{x_{1}}\psi_{N,\hbar}
-(\operatorname{div} u)\ol{\psi_{N,\hbar}}\Delta_{x_{1}}\psi_{N,\hbar}dX_{N}\notag\\
=&-\hbar^{2}\sum_{j,k=1}^{3}\operatorname{Re}\int 2\pa_{k}u^{j}(\pa_{k}\ol{\psi_{N,\hbar}})\pa_{j}\psi_{N,\hbar}+
2u^{j}(\pa_{k}\ol{\psi_{N,\hbar}})\pa_{k}\pa_{j}\psi_{N,\hbar}dX_{N}\notag\\
&+\hbar^{2}\operatorname{Re}\int (\operatorname{div} u)\ol{\psi_{N,\hbar}}\Delta_{x_{1}}\psi_{N,\hbar}dX_{N}\notag\\
=&\hbar^2\sum_{j,k=1}^{3}\int \lrs{\pa_{j}u^{k}+\pa_{k}u^{j}}
\partial_j\psi_{N,\hbar}\partial_k\overline{\psi_{N,\hbar}}dX_N\notag\\
&-\hbar^{2}\sum_{j,k=1}^{3}\operatorname{Re}\int 2u^{j}(\pa_{k}\ol{\psi_{N,\hbar}})\pa_{k}\pa_{j}\psi_{N,\hbar}dX_{N}+\hbar^{2}\operatorname{Re}\int (\operatorname{div} u)\ol{\psi_{N,\hbar}}\Delta_{x_{1}}\psi_{N,\hbar}dX_{N},\label{equ:evolution,M2,kinetic,small term}
\end{align}
where we used the notations $u=(u^{1},u^{2},u^{3})$, $x_{1}=(x_{1}^1,x_{1}^2,x_{1}^3)\in\mathbb{R}^3$ and $\partial_j=\partial_{x_{1}^j}$.

Using again integration by parts on the two terms of \eqref{equ:evolution,M2,kinetic,small term} gives
\begin{align}\label{equ:evolution,M2,kinetic,error term}
&-\hbar^{2}\sum_{j,k=1}^{3}\operatorname{Re}\int 2u^{j}(\pa_{k}\ol{\psi_{N,\hbar}})\pa_{k}\pa_{j}\psi_{N,\hbar}dX_{N}+\hbar^{2}\operatorname{Re}\int (\operatorname{div} u)\ol{\psi_{N,\hbar}}\Delta_{x_{1}}\psi_{N,\hbar}dX_{N}\\
=&\hbar^{2}\operatorname{Re}\int \operatorname{div} u\lrs{|\nabla_{x_{1}}\psi_{N,\hbar}|^{2}+\ol{\psi_{N,\hbar}}\Delta_{x_{1}}\psi_{N,\hbar}}dX_{N}\notag\\
=&\frac{\hbar^{2}}{2}\operatorname{Re}\int \operatorname{div} u(\Delta_{x_{1}} |\psi_{N,\hbar}|^{2})dX_{N}\notag\\
=&\frac{\hbar^{2}}{2}\int (\Delta \operatorname{div} u)(t,x_{1}) \rho_{N,\hbar}^{(1)}(t,x_{1})dx_{1}.\notag
\end{align}
Combining estimates \eqref{equ:estimate,evolution,M2}--\eqref{equ:evolution,M2,kinetic,error term}, we provide
\begin{align*}
&\frac{d}{dt}\mathcal{M}_2(t)\\
=&-2\lra{\pa_{t}u,J_{N,\hbar}^{(1)}}+
\hbar^2\sum_{j,k=1}^{3}\int \lrs{\pa_{j}u^{k}+\pa_{k}u^{j}}
\partial_j\psi_{N,\hbar}\partial_k\overline{\psi_{N,\hbar}}dX_N\\
&+\frac{\hbar^2}{2}\int (\Delta \operatorname{div} u)(t,x_{1})\rho_{N,\hbar}^{(1)}(t,x_{1})dx_{1}\\
&+\frac{(N-1)}{N}\int (u(t,x_{1})-u(t,x_{2}))\cdot \nabla_{x_{1}}(V_N+V_{c})(x_{1}-x_{2})\rho_{N,\hbar}^{(2)}(t,x_{1},x_{2})dx_{1}dx_{2}.
\end{align*}

For $\mathcal{M}_3(t)$, by the Euler-Poisson equation \eqref{equ:euler equation} and the mass conservation law \eqref{equ:mass conservation}, we obtain
\begin{align*}
&\frac{d}{dt}\mathcal{M}_3(t)\\
=&\frac{d}{dt}\int \vert u(t,x_{1})\vert^2\rho_{N,\hbar}^{(1)}(t,x_{1})dx_{1}\nonumber\\
=&\int 2u(t,x_{1})\cdot\partial_t u(t,x_{1})\rho_{N,\hbar}^{(1)}(t,x_{1})dx_{1}+\int \vert u(t,x_{1})\vert^2\partial_t\rho_{N,\hbar}^{(1)}(t,x_{1})dx_{1}\notag\\
=&-2\int u(t,x_{1})\cdot\left(u\cdot\nabla u+b_{0}\nabla\rho+\nabla V_{c}*\rho\right)\rho_{N,\hbar}^{(1)}(t,x_{1})dx_{1}\\
&+\int \nabla\left(\vert u(t,x_{1})\vert^2\right)J_{N,\hbar}^{(1)}(t,x_{1})dx_{1}.\notag
\end{align*}
Expanding it gives
\begin{align*}
\frac{d}{dt}\mathcal{M}_3(t)=&-2\sum_{j,k=1}^{3}\int u^{k}u^{j}\pa_{j}u^{k}\rho_{N,\hbar}^{(1)}(t,x_{1})dx_{1}
-2b_{0}\lra{\rho_{N,\hbar}^{(1)},u\cdot \nabla \rho}
-2\lra{\rho_{N,\hbar}^{(1)},u\cdot \nabla V_{c}*\rho}\\
&+2\sum_{j,k=1}^{3}\int u^{k}\pa_{j}u^{k} J_{N,\hbar}^{(1)}(t,x_{1})dx_{1}.
\end{align*}

For $\mathcal{M}_4(t)$, plugging in the Euler-Poisson equation \eqref{equ:euler equation}, we have
\begin{align*}
\frac{d}{dt}\mathcal{M}_{4}(t)=&b_{0}\frac{d}{dt}\int_{\mathbb{R}^{3}}\rho(t,x_{1})\rho(t,x_{1})dx_{1}+
\frac{d}{dt}\int_{\mathbb{R}^{3}}\rho(t,x_{1})(V_{c}*\rho)(t,x_{1})dx_{1},\\
=&2b_{0}\lra{\rho,u\cdot \nabla \rho}+2\lra{\rho,u\cdot \nabla V_{c}*\rho}.
\end{align*}

For $\mathcal{M}_{5}(t)$, similarly we get to
\begin{align*}
&\frac{d}{dt}\mathcal{M}_{5}(t)\\
=&-2b_{0}\frac{d}{dt}\int_{\mathbb{R}^3}\rho(t,x_{1})\rho_{N,\hbar}^{(1)}(t,x_{1})dx_{1}-2
\frac{d}{dt}\int (V_{c}*\rho)(t,x_{1})\rho_{N,\hbar}^{(1)}(t,x_{1})dx_{1}\\
=&-2b_{0}\lrs{\lra{\pa_{t}\rho,\rho_{N,\hbar}^{(1)}}+\lra{\rho,\pa_{t}\rho_{N,\hbar}^{(1)}}}-
2\lrs{\lra{\pa_{t}\rho,V_{c}*\rho_{N,\hbar}^{(1)}}+\lra{V_{c}*\rho,\pa_{t}\rho_{N,\hbar}^{(1)}}}.
\end{align*}
Plugging in the Euler-Poisson equation \eqref{equ:euler equation} and the mass conservation law \eqref{equ:mass conservation}, we have
\begin{align*}
&\frac{d}{dt}\mathcal{M}_{5}(t)\\
=&2b_{0}\lrs{\lra{\operatorname{div} (\rho u),\rho_{N,\hbar}^{(1)}}+\lra{\rho,\operatorname{div}J_{N,\hbar}^{(1)}}}
+2\lrs{\lra{\operatorname{div} (\rho u),V_{c}*\rho_{N,\hbar}^{(1)}}+\lra{V_{c}*\rho,\operatorname{div}J_{N,\hbar}^{(1)}}}\\
=&-2b_{0}\lra{\rho ,u\cdot \nabla \rho_{N,\hbar}^{(1)}}-2\lra{\rho ,u\cdot \nabla V_{c}*\rho_{N,\hbar}^{(1)}}-2b_{0}\lra{\nabla\rho,J_{N,\hbar}^{(1)}}
-2\lra{\nabla V_{c}*\rho,J_{N,\hbar}^{(1)}}.
\end{align*}

Therefore, putting the five terms together, we reach
\begin{align}
\frac{d}{dt}\mathcal{M}(t)=&\frac{d}{dt}\mathcal{M}_1(t)+\frac{d}{dt}\mathcal{M}_2(t)+\frac{d}{dt}\mathcal{M}_3(t)
+\frac{d}{dt}\mathcal{M}_4(t)+\frac{d}{dt}\mathcal{M}_5(t)\notag\\
=&-2\lra{\pa_{t}u,J_{N,\hbar}^{(1)}}-\hbar^2\sum_{j,k=1}^{3}\int_{\mathbb{R}^{3N}}\lrs{\pa_{j}u^{k}+\pa_{k}u^{j}}
\partial_j\psi_{N,\hbar}\partial_k\overline{\psi_{N,\hbar}}dX_N \label{equ:evolution of modulated energy,sumjk,first}\\
&+\frac{\hbar^2}{2}\int_{\mathbb{R}^{3}}\Delta(\operatorname{div} u)(t,x_{1})\rho_{N,\hbar}^{(1)}(t,x_{1})dx_{1}\notag\\
&+\frac{N-1}{N}\int (u(t,x_{1})-u(t,x_{2}))\cdot \nabla_{x_{1}}(V_N+V_{c})(x_{1}-x_{2})\rho_{N,\hbar}^{(2)}(t,x_{1},x_{2})dx_{1}dx_{2}\notag\\
&-2\sum_{j,k=1}^{3}\int u^{k}u^{j}\pa_{j}u^{k}\rho_{N,\hbar}^{(1)}(t,x_{1})dx_{1}
-2b_{0}\lra{\rho_{N,\hbar}^{(1)},u\cdot \nabla \rho}
-2\lra{\rho_{N,\hbar}^{(1)},u\cdot \nabla V_{c}*\rho}\label{equ:evolution of modulated energy,sumjk,second}\\
&+\lra{\nabla (|u|^{2}),J_{N,\hbar}^{(1)}}+2b_{0}\lra{\rho,u\cdot \nabla \rho}+2\lra{\rho,u\cdot \nabla V_{c}*\rho}\notag\\
&-2b_{0}\lra{\rho ,u\cdot \nabla \rho_{N,\hbar}^{(1)}}-2\lra{\rho ,u\cdot \nabla V_{c}*\rho_{N,\hbar}^{(1)}}-2b_{0}\lra{\nabla\rho,J_{N,\hbar}^{(1)}}
-2\lra{\nabla V_{c}*\rho,J_{N,\hbar}^{(1)}}.\notag
\end{align}
From the above equation, we collect the $\delta$-type potential contribution part $\wt{\mathcal{F}}_{\delta}(t)$ in \eqref{equ:evolution of modulated energy} from
\begin{align*}
&\frac{N-1}{N}\int (u(t,x_{1})-u(t,x_{2}))\cdot \nabla_{x_{1}}V_N(x_{1}-x_{2})\rho_{N,\hbar}^{(2)}(t,x_{1},x_{2})dx_{1}dx_{2}\\
&-2b_{0}\lra{\rho_{N,\hbar}^{(1)},u\cdot \nabla \rho}-2b_{0}\lra{\rho,u\cdot \nabla \rho_{N,\hbar}^{(1)}}
+2b_{0}\lra{\rho,u\cdot\nabla  \rho }\\
=&\frac{N-1}{N}\int (u(t,x_{1})-u(t,x_{2}))\cdot \nabla_{x_{1}}V_N(x_{1}-x_{2})\rho_{N,\hbar}^{(2)}(t,x_{1},x_{2})dx_{1}dx_{2}\\
&-b_{0}\int \operatorname{div} u(t,x_{1})\rho(t,x_{1})\lrc{\rho(t,x_{1})-2\rho_{N,\hbar}^{(1)}(t,x_{1})}dx_{1}.
\end{align*}
and the Coulomb potential contribution part $\wt{\mathcal{F}}_{c}(t)$ in \eqref{equ:evolution of modulated energy} from
\begin{align*}
&\frac{N-1}{N}\int (u(t,x_{1})-u(t,x_{2}))\cdot \nabla_{x_{1}}V_{c}(x_{1}-x_{2})\rho_{N,\hbar}^{(2)}(t,x_{1},x_{2})dx_{1}dx_{2}\\
&-2\lra{\rho_{N,\hbar}^{(1)},u\cdot \nabla V_{c}*\rho}-2\lra{\rho,u\cdot \nabla V_{c}* \rho_{N,\hbar}^{(1)}}
+2\lra{\rho,u\cdot\nabla  V_{c}* \rho }\\
=&\int(u(t,x_{1})-u(t,x_{2}))\nabla V_c(x_{1}-x_{2})\\
&\left[\frac{N-1}{N}\rho_{N,\hbar}^{(2)}(t,x_{1},x_{2})+\rho(t,x_{1})\rho(t,x_{2})-
2\rho(t,x_{1})\rho_{N,\hbar}^{(1)}(t,x_{2})\right]dx_{1}dx_{2},
\end{align*}
where in the last equality we used the antisymmetry of $\nabla V_{c}$.

As for the first term in \eqref{equ:evolution of modulated energy}, we use the Euler-Poisson equation \eqref{equ:euler equation} to combine the terms taking the form of $\lra{\bullet,J_{N,\hbar}^{(1)}}$
\begin{align}
&-2\lra{\pa_{t}u,J_{N,\hbar}^{(1)}}-2b_{0}\lra{\nabla\rho,J_{N,\hbar}^{(1)}}
-2\lra{\nabla V_{c}*\rho,J_{N,\hbar}^{(1)}}+\lra{\nabla (|u|^{2}),J_{N,\hbar}^{(1)}}\notag\\
=&2\lra{u\cdot \nabla u,J_{N,\hbar}^{(1)}}+2\lra{\nabla\cdot (u\otimes u),J_{N,\hbar}^{(1)}}\notag\\
=&i\hbar\sum_{j,k=1}^{3}\int_{\mathbb{R}^{3N}}\lrs{\pa_{j}u^{k}+\pa_{k}u^{j}}u^{j}(\psi_{N,\hbar}\partial_k\overline{\psi_{N,\hbar}}-\overline{\psi_{N,\hbar}}
\partial_k\psi_{N,\hbar})dX_N.\label{equ:evolution of modulated energy,sumjk,third}
\end{align}
If we rewrite the first term on the right hand side of \eqref{equ:evolution of modulated energy}
\begin{align*}
&-\sum_{j,k=1}^{3}\int \lrs{\pa_{j}u^{k}+\pa_{k}u^{j}}(-i\hbar\partial_j\psi_{N,\hbar}-u^{j}\psi_{N,\hbar})
\overline{(-i\hbar\partial_k\psi_{N,\hbar}-u^{k}\psi_{N,\hbar})}dX_N\\
=&-\hbar^2\sum_{j,k=1}^{3}\int \lrs{\pa_{j}u^{k}+\pa_{k}u^{j}}
\partial_j\psi_{N,\hbar}\partial_k\overline{\psi_{N,\hbar}}dX_N\\
&-2\sum_{j,k=1}^{3}\int u^{k}u^{j}\pa_{j}u^{k}\rho_{N,\hbar}^{(1)}(t,x_{1})dx_{1}\\
&+i\hbar\sum_{j,k=1}^{3}\int_{\mathbb{R}^{3N}}\lrs{\pa_{j}u^{k}+\pa_{k}u^{j}}u^{j}(\psi_{N,\hbar}\partial_k\overline{\psi_{N,\hbar}}-\overline{\psi_{N,\hbar}}
\partial_k\psi_{N,\hbar})dX_N,
\end{align*}
these are the $\sum_{j,k}^{3}$ terms in \eqref{equ:evolution of modulated energy,sumjk,first},\eqref{equ:evolution of modulated energy,sumjk,second} and \eqref{equ:evolution of modulated energy,sumjk,third}. Therefore, we arrive at equation \eqref{equ:evolution of modulated energy} and complete the proof.

\end{proof}

\section{$(H_{N,\hbar})^{2}$ Energy Estimate Using Singular Correlation Structure}\label{section:Energy Estimate Using Singular Correlation Structure}

As mentioned in the preliminary reduction step in the outline, Section \ref{section:Outline of the Proof}, the two-body energy estimate is crucial for the analysis of the $\delta$-type potential parts $\mathcal{F}_{\delta}(t)$ and $\wt{\mathcal{F}}_{\delta}(t)$.
The main difficulty is the singularities simultaneously from the Coulomb potential, from the direct  $\delta$-potential in the $N\to \infty$ limit, and from the interparticle singular correlation structure.

Recall the zero-energy scattering equation
\begin{equation}\label{equ:scattering equation}
	\left\{
	\begin{aligned}
		&\lrs{-\hbar^{2}\Delta+\frac{1}{N}V_{N}(x)}(1-w_{N,\hbar}(x))=0,\\
		&\lim_{|x|\to \infty} w_{N,\hbar}(x)=0.
	\end{aligned}
	\right.
\end{equation}
and our target estimate
\begin{align}\label{equ:target estimate,two-body estimate}
\Big\langle(1-\hbar^{2}\Delta_{x_{1}})(1-\hbar^{2}\Delta_{x_{2}})\frac{\psi_{N,\hbar}(t,X_{N})}{1-w_{N,\hbar}(x_{1}-x_{2})},
\frac{\psi_{N,\hbar}(t,X_{N})}{1-w_{N,\hbar}(x_{1}-x_{2})}\Big\rangle\leq C.
\end{align}

We first give the properties of the scattering function.
\begin{lemma}\label{lemma:scattering equation}
Suppose that $V\geq 0$ is smooth, spherical symmetric with compact support and $1-w_{N,\hbar}(x)$
satisfies the scattering equation \eqref{equ:scattering equation}.
Then there exists $C$, depending on $V$, such that
\begin{align}
&0 \leq w_{N,\hbar}(x)\leq \frac{C}{N\hbar^{2}(|x|+N^{-\be})},\\
&|\nabla w_{N,\hbar}(x)|\leq \frac{C}{N\hbar^{2}(|x|^{2}+N^{-2\be})},
\label{equ:gradient property,scattering function}
\end{align}
for all $x\in \mathbb{R}^{3}$.

\end{lemma}
\begin{proof}
The properties of scattering function have been studied by many authors, see, for example, \cite{BCS17,ESY06,LSSY05}. Here, we include a proof for completeness.
First, by the maximum principle, it follows that $(1-w_{N,\hbar}(x))\leq 1$.
From the scattering equation \eqref{equ:scattering equation,appendix}, we can rewrite
\begin{align}\label{equ:rewritten form,scattering function}
w_{N,\hbar}(x)=\frac{c_{0}}{2N\hbar^{2}}\int \frac{1}{|x-y|}V_{N}(y)(1-w_{N,\hbar}(y))dy,
\end{align}
where $c_{0}$ is the renormalized constant.

Then the Hardy-Littlewood-Sobolev inequality implies that
\begin{align*}
(|x|+N^{-\be})w_{N,\hbar}(x)=&\frac{c_{0}}{N\hbar^{2}}\int \frac{|x|+N^{-\be}}{|x-y|}V_{N}(y)(1-w_{N,\hbar}(y))dy\\
\leq&\frac{c_{0}}{N\hbar^{2}}\int \frac{|x-y|+|y|+N^{-\be}}{|x-y|}V_{N}(y)dy\\
=&\frac{c_{0}\n{V}_{L^{1}}}{N\hbar^{2}}+\frac{c_{0}}{N^{1+\be}\hbar^{2}}\int \frac{1+N^{\be}|y|V_{N}(y)}{|x-y|}dy\\
\lesssim & \frac{1}{N\hbar^{2}}\lrs{\n{V}_{L^{1}}+\n{\lra{y}V(y)}_{L^{\frac{3}{2}}}}.
\end{align*}

For \eqref{equ:gradient property,scattering function}, by taking the gradient of \eqref{equ:rewritten form,scattering function}, we also have
\begin{align*}
&|(|x|^{2}+N^{-2\be})\nabla_{x} w_{N,\hbar}(x)|\\
\leq& \frac{1}{2N\hbar^{2}}
(|x|^{2}+N^{-2\be})\bbabs{\int \nabla_{x}\frac{1}{|x-y|}V_{N}(y)(1-w_{N,\hbar}(y))dy}\\
\leq& \frac{1}{N\hbar^{2}}\int \frac{|x-y|^{2}+|y|^{2}+N^{-2\be}}{|x-y|^{2}}V_{N}(y)dy\\
\leq& \frac{\n{V}_{L^{1}}}{N\hbar^{2}}+\frac{1}{N\hbar^{2}}
\int \frac{1+N^{2\be}|y|^{2}V_{N}(y)}{|x-y|^{2}}dy\\
\leq& \frac{1}{N\hbar^{2}}\lrs{\n{V}_{L^{1}}+\n{\lra{y}^{2}V(y)}_{L^{3}}}.
\end{align*}
\end{proof}

For simplicity,
we adopt the shorthands
\begin{align}
	w_{12}=w_{N,\hbar}(x_{1}-x_{2}),\quad \nabla w_{12}=(\nabla w_{N,\hbar})(x_{1}-x_{2}),
\end{align}
and start the proof of \eqref{equ:target estimate,two-body estimate}.
\begin{lemma}\label{lemma:H2 energy estimate,lower bound}
	Let $\beta\in (0,1)$ and $N^{\be-1}\hbar^{-2}\ll 1$. Then we have
	\begin{align}\label{equ:H2 energy estimate,lower bound}
		\lra{\psi,(H_{N,\hbar}+N)^{2}\psi}\geq \frac{N(N-1)}{16}\Big\langle(1-\hbar^{2}\Delta_{x_{1}})(1-\hbar^{2}\Delta_{x_{2}})\frac{\psi}
{1-w_{12}},\frac{\psi}{1-w_{12}}\Big\rangle
	\end{align}
	for $\psi\in L_{s}^{2}(\mathbb{R}^{3N})$.
\end{lemma}
\begin{proof}
Let
	\begin{align}\label{equ:Ti}
		T_{i}:=1-\frac{\hbar^{2}}{2}\Delta_{i}+\frac{1}{2N}\sum_{j:j\neq i}V_{N}(x_{i}-x_{j})+\frac{1}{2N}\sum_{j:j\neq i}V_{c}(x_{i}-x_{j}),
	\end{align}
we rewrite the Hamiltonian \eqref{equ:hamiltonian}
	\begin{align*}
		H_{N,\hbar}+N=\sum_{i=1}^{N}T_{i}.
	\end{align*}
	By the symmetry of $\psi$, we have
	\begin{align}\label{equ:H2 energy estimate,T1T2}
		\lra{\psi,(H_{N,\hbar}+N)^{2}\psi}
		=&\sum_{i,j}^{N}\lra{\psi,T_{i}T_{j}\psi}\\
		=& N(N-1)\lra{\psi,T_{1}T_{2}\psi}+N\lra{\psi,T_{1}^{2}\psi}\notag\\
		\geq& N(N-1)\lra{\psi,T_{1}T_{2}\psi}.\notag
	\end{align}
	Note that $\psi=(1-w_{12})\phi_{12}$ and we have
	\begin{align*}
		-\hbar^{2}\Delta_{1}\psi=&-\hbar^{2}\Delta_{1}[(1-w_{12})\phi_{12}]\\
		=&(1-w_{12})(-\hbar^{2}\Delta_{1}\phi_{12})+2\hbar \nabla_{1} w_{12}
		\hbar\nabla_{1}\phi_{12}+\hbar^{2}\Delta_{1} w_{12}\phi_{12}.
	\end{align*}
	Thus, together with the scattering equation \eqref{equ:scattering equation}, we arrive at
	\begin{align}\label{equ:T1}
		T_{1}\psi=
		&T_{1}[(1-w_{12})\phi_{12}]\\
		=&(1-w_{12})(\phi_{12}-\frac{\hbar^{2}}{2}\Delta_{1}\phi_{12})+\hbar \nabla_{1} w_{12}
		\hbar\nabla_{1}\phi_{12}+\frac{\hbar^{2}}{2}\Delta_{1} w_{12}\phi_{12}\notag\\
		&+(1-w_{12})\lrc{\frac{1}{2N}\sum_{j\geq 2}V_{N}(x_{1}-x_{j})\phi_{12}+
			\frac{1}{2N}\sum_{j\geq 2}V_{c}(x_{1}-x_{j})\phi_{12}}\notag\\
		=&(1-w_{12})\lrc{\phi_{12}-\frac{\hbar^{2}}{2}\Delta_{1}\phi_{12}+\frac{\hbar \nabla_{1} w_{12}}{1-w_{12}}
			\hbar\nabla_{1}\phi_{12}}\notag\\
		&+(1-w_{12})\lrc{\frac{1}{2N}\sum_{j\geq 3}V_{N}(x_{1}-x_{j})\phi_{12}+
			\frac{1}{2N}\sum_{j\geq 2}V_{c}(x_{1}-x_{j})\phi_{12}}.\notag
	\end{align}
	Similarly, we also have
	\begin{align}\label{equ:T2}
		T_{2}\psi=&(1-w_{12})\lrc{\phi_{12}-\frac{\hbar^{2}}{2}\Delta_{2}\phi_{12}+\frac{\hbar \nabla w_{12}}{1-w_{12}}
			\hbar\nabla_{2}\phi_{12}}\\
		&+(1-w_{12})\lrc{\frac{1}{2N}\sum_{j\geq 3}V_{N}(x_{2}-x_{j})\phi_{12}+
			\frac{1}{2N}\sum_{j\neq  2}V_{c}(x_{2}-x_{j})\phi_{12}}.\notag
	\end{align}
Further define the shorthands
	\begin{align}
		&L_{1}:=1-\frac{\hbar^{2}}{2}\Delta_{1}+\frac{\hbar \nabla_{1} w_{12}}{1-w_{12}}\hbar \nabla_{1},\label{equ:L1}\\
		&L_{2}:=1-\frac{\hbar^{2}}{2}\Delta_{2}+\frac{\hbar \nabla_{2} w_{12}}{1-w_{12}}\hbar \nabla_{2},\label{equ:L2}
	\end{align}
	which are symmetric with respect to the measure $(1-w_{12})^{2}dx$, that is,
	\begin{align}\label{equ:symmetry property,L}
		\int (1-w_{12})^{2}\ol{f}(L_{1}g)=&\int(1-w_{12})^{2}(\ol{L_{1}f})g\\
		=&\int (1-w_{12})^{2}\lrc{\ol{f}g+\frac{\hbar^{2}}{2} \nabla_{1}\ol{f}
			\nabla_{1}g}.\notag
	\end{align}
	Therefore, from \eqref{equ:T1} and \eqref{equ:T2} we obtain
	\begin{align*}
		&\lra{T_{1}\psi,T_{2}\psi}\\
		=&\int (1-w_{12})^{2} \lrs{L_{1}+\frac{1}{2N}\sum_{j\geq 3}(V_{N}+V_{c})(x_{1}-x_{j})+
			\frac{1}{2N}V_{c}(x_{1}-x_{2})}\ol{\phi}_{12}\\
		&\cdot \lrs{L_{2}+\frac{1}{2N}\sum_{j\geq 3}(V_{N}+V_{c})(x_{2}-x_{j})+\frac{1}{2N}V_{c}(x_{1}-x_{2})}\phi_{12}.
	\end{align*}
	Expanding it gives
	\begin{align}\label{equ:T1,T2}
		&\lra{T_{1}\psi,T_{2}\psi}\\
		= &\int(1-w_{12})^{2}L_{1}\ol{\phi}_{12}L_{2}\phi_{12}\notag\\
		&+\int(1-w_{12})^{2}(L_{1}\ol{\phi}_{12})\lrc{\frac{1}{2N}\sum_{j\geq 3}(V_{N}+V_{c})(x_{2}-x_{j})+\frac{1}{2N}
			V_{c}(x_{1}-x_{2})}\phi_{12}\notag\\
		&+\int(1-w_{12})^{2}\lrc{\frac{1}{2N}\sum_{j\geq 3}(V_{N}+V_{c})(x_{1}-x_{j})+\frac{1}{2N}
			V_{c}(x_{1}-x_{2})}\ol{\phi}_{12}L_{2}\phi_{12}\notag\\
		&+\int(1-w_{12})^{2}\lrc{\frac{1}{2N}\sum_{j\geq 3}(V_{N}+V_{c})(x_{1}-x_{j})+\frac{1}{2N}
			V_{c}(x_{1}-x_{2})}\ol{\phi}_{12}\notag\\
		&\quad\quad\quad\quad\quad\quad\quad\cdot \lrc{\frac{1}{2N}\sum_{j\geq 3}(V_{N}+V_{c})(x_{1}-x_{j})+\frac{1}{2N}
			V_{c}(x_{1}-x_{2})}\ol{\phi}_{12}.\notag
	\end{align}
	By the nonnegativity of the potentials, we can discard the last term on the r.h.s of \eqref{equ:T1,T2}. The symmetry property \eqref{equ:symmetry property,L} of the operators $L_{1}$ and $L_{2}$ then yields
	\begin{align*}
		&\lra{T_{1}\psi,T_{2}\psi}\\
		\geq &\int(1-w_{12})^{2}L_{1}\ol{\phi}_{12}L_{2}\phi_{12}\\
		&+\int(1-w_{12})^{2}\lrs{|\phi_{12}|^{2}+\frac{\hbar^{2}}{2}|\nabla_{1}\phi_{12}|^{2}}\frac{1}{2N}\sum_{j\geq 3}(V_{N}+V_{c})(x_{2}-x_{j})\\
		&+\int(1-w_{12})^{2}|\phi_{12}|^{2}\frac{1}{2N}V_{c}(x_{1}-x_{2})+\frac{\hbar^{2}}{2}\int (1-w_{12})^{2}
		\nabla_{1}\ol{\phi}_{12} \nabla_{1}\lrs{\frac{1}{2N}
			V_{c}(x_{1}-x_{2})\phi_{12}}\\
		&+\int(1-w_{12})^{2}\lrs{|\phi_{12}|^{2}+\frac{\hbar^{2}}{2}|\nabla_{2}\phi_{12}|^{2}}\frac{1}{2N}\sum_{j\geq 3}(V_{N}+V_{c})(x_{2}-x_{j})\\
		&+\int(1-w_{12})^{2}|\phi_{12}|^{2}\frac{1}{2N}V_{c}(x_{1}-x_{2})+\frac{\hbar^{2}}{2}\int(1-w_{12})^{2} \nabla_{2}\lrs{\frac{1}{2N}
			V_{c}(x_{1}-x_{2}) \ol{\phi}_{12}}\nabla_{2}\phi_{12}.
	\end{align*}
	Again using the nonnegativity of the potentials, we reach
	\begin{align}\label{equ:T1,T2,reduced form}
		\lra{T_{1}\psi,T_{2}\psi}\geq&\int(1-w_{12})^{2}L_{1}\ol{\phi}_{12}L_{2}\phi_{12}\\
		&+\frac{\hbar^{2}}{2}\int (1-w_{12})^{2}
		\nabla_{1}\ol{\phi}_{12} \nabla_{1}\lrs{\frac{1}{2N}
			V_{c}(x_{1}-x_{2})\phi_{12}}\notag\\
		&+\frac{\hbar^{2}}{2}\int(1-w_{12})^{2} \nabla_{2}\lrs{\frac{1}{2N}
			V_{c}(x_{1}-x_{2}) \ol{\phi}_{12}}\nabla_{2}\phi_{12}\notag\\
		=&I+II+III.\notag
	\end{align}
For the first term $I$ on the r.h.s of \eqref{equ:T1,T2,reduced form}, by \eqref{equ:symmetry property,L}, we have
	\begin{align}\label{equ:I,estimate}
		I=&\int(1-w_{12})^{2}\lrc{\ol{\phi}_{12}L_{2}\phi_{12}+\frac{\hbar^{2}}{2}\nabla_{1}\ol{\phi}_{12}\nabla_{1}L_{2}\phi_{12}}\\
		=&\int(1-w_{12})^{2}\lrc{|\phi_{12}|^{2}+\frac{\hbar^{2}}{2}|\nabla_{2}\phi_{12}|^{2}+\frac{\hbar^{2}}{2}|\nabla_{1}\phi_{12}|^{2}+\frac{\hbar^{4}}{4}|\nabla_{1}\nabla_{2}\phi_{12}|^{2}}\notag\\
		&+\frac{\hbar^{2}}{2}
		\int(1-w_{12})^{2}\nabla_{1}\ol{\phi}_{12}[\nabla_{1},L_{2}]\phi_{12}\notag\\
		\geq&\frac{1}{2}\int\lrc{|\phi_{12}|^{2}+\frac{\hbar^{2}}{2}|\nabla_{2}\phi_{12}|^{2}+\frac{\hbar^{2}}{2}|\nabla_{1}\phi_{12}|^{2}+\frac{\hbar^{4}}{4}|\nabla_{1}\nabla_{2}\phi_{12}|^{2}}\notag\\
		&+\frac{\hbar^{2}}{2}
		\int(1-w_{12})^{2}\nabla_{1}\ol{\phi}_{12}[\nabla_{1},L_{2}]\phi_{12},\notag
	\end{align}
	where in the last inequality we have used Lemma \ref{lemma:scattering equation} that $(1-w_{12})^{2}\geq \frac{1}{2}$.
	To control the last term on the r.h.s of \eqref{equ:I,estimate},
	we note that
	\begin{align}
		|[\nabla_{1},L_{2}]|=\hbar^{2}\bbabs{\lrc{\nabla_{1},\frac{\nabla_{2} w_{12}}{1-w_{12}}}}\leq
		\hbar^{2}\lrc{\frac{|\nabla^{2}w_{12}|}{1-w_{12}}+\lrs{\frac{\nabla w_{12}}{1-w_{12}}}^{2}}.\notag
	\end{align}
	Therefore, we have
	\begin{align*}
		\frac{\hbar^{2}}{2}
		\int(1-w_{12})^{2}\nabla_{1}\ol{\phi}_{12}[\nabla_{1},L_{2}]\phi_{12}
		\leq& \frac{\hbar^{4}}{2}\int \lrs{|\nabla^{2}w_{12}|+|\nabla w_{12}|^{2}}|\nabla_{1}\phi_{12}||\phi_{12}|\\
		=&I_{1}+I_{2}.
	\end{align*}
	
	For $I_{1}$, by H\"{o}lder and Sobolev inequalities we have
	\begin{align}
		I_{1}\leq& \hbar^{4} \n{\nabla^{2}w_{12}}_{L_{x_{2}}^{\frac{3}{2}}}\n{\nabla_{1}\phi_{12}}_{L^{2}L_{x_{2}}^{6}}\n{\phi_{12}}_{L^{2}L_{x_{2}}^{6}}\notag\\
		\lesssim& \hbar^{4} \n{\nabla^{2}w_{12}}_{L_{x_{2}}^{\frac{3}{2}}}\n{\nabla_{1}\nabla_{2}\phi_{12}}_{L^{2}L_{x_{2}}^{2}}\n{\nabla_{2}\phi_{12}}_{L^{2}L_{x_{2}}^{2}}.\notag
	\end{align}
	By the Calder\'{o}n-Zygmund theory which implies that $\n{\nabla^{2} f}_{L^{p}}\lesssim \n{\Delta f}_{L^{p}}$ for $1<p<\infty$ and the scattering equation
	\eqref{equ:scattering equation}, we get
	\begin{align}
		\hbar^{2}\n{\nabla^{2}w_{12}}_{L_{x_{2}}^{\frac{3}{2}}}\lesssim \hbar^{2}\n{\Delta w_{12}}_{L_{x_{2}}^{\frac{3}{2}}}\leq \frac{1}{N}\n{V_{N}(x_{1}-x_{2})}_{L_{x_{2}}^{\frac{3}{2}}}\lesssim \frac{N^{\be}}{N}.\notag
	\end{align}
	Thus, we arrive at
	\begin{align}\label{equ:I1,estimate}
		I_{1}\lesssim \frac{N^{\be}}{N}\hbar^{2}\lrs{\n{\nabla_{1}\nabla_{2}\phi_{12}}_{L^{2}}^{2}+
			\n{\nabla_{1}\phi_{12}}_{L^{2}}^{2}}.
	\end{align}
	
	For $I_{2}$, by the properties of the scattering function in Lemma \ref{lemma:scattering equation}, we have
	\begin{align*}
		|\nabla w_{12}|\leq \frac{C}{N\hbar^{2}(|x_{1}-x_{2}|^{2}+N^{-2\be})},
	\end{align*}
	which implies that
	\begin{align*}
		|\nabla w_{12}|^{2}\lesssim \frac{N^{2\be}}{N\hbar^{2}}\frac{1}{N\hbar^{2}|x_{1}-x_{2}|^{2}}=\frac{N^{2\be}}{N^{2}\hbar^{4}|x_{1}-x_{2}|^{2}}.
	\end{align*}
	Then by Cauchy-Schwarz and Hardy's inequalities, we get
	\begin{align}\label{equ:I2,estimate}
		I_{2}\leq& \hbar^{4} \int |\nabla w_{12}|^{2} (|\nabla_{1}\phi_{12}|^{2}
		+|\phi_{12}|^{2})\\
		\lesssim& \frac{N^{2\be}}{N^{2}}\lrs{\n{\nabla_{1}\nabla_{2}\phi_{12}}_{L^{2}}^{2}+
			\n{\nabla_{2}\phi_{12}}_{L^{2}}^{2}}.\notag
	\end{align}

	Next, we deal with the terms $II$ and $III$ in \eqref{equ:I,estimate}. For $II$, we have
	\begin{align*}
		II=&\frac{\hbar^{2}}{2}\int (1-w_{12})^{2}
		\nabla_{1}\ol{\phi}_{12} \nabla_{1}\lrs{\frac{1}{2N}
			V_{c}(x_{1}-x_{2})\phi_{12}}\\
		=& \frac{\hbar^{2}}{2}\int (1-w_{12})^{2}
		\lrc{|\nabla_{1}\phi_{12}|^{2} \frac{1}{2N}
			V_{c}(x_{1}-x_{2})+(\nabla_{1}\ol{\phi}_{12})\phi_{12}
			\nabla_{1}\frac{1}{2N}
			V_{c}(x_{1}-x_{2})}\\
		\geq&\frac{\hbar^{2}}{2}\int (1-w_{12})^{2}
		(\nabla_{1}\ol{\phi}_{12})\phi_{12}
		\nabla_{1}\frac{1}{2N}
		V_{c}(x_{1}-x_{2}),
	\end{align*}
	where in the last inequality we used the positivity of the Coulomb potential. Noting that $|\nabla V_{c}(x)|\lesssim |x|^{-2}$, we can use Cauchy-Schwarz and Hardy's inequalities to obtain
	\begin{align}\label{equ:II,estimate}
		II\geq& -\frac{\hbar^{2}}{2N}\int \lrs{|\nabla_{1}\phi_{12}|^{2}+|\phi_{12}|^{2}}\frac{1}{|x_{1}-x_{2}|^{2}}\\
		\gtrsim & -\frac{\hbar^{2}}{N}\lrs{\n{\nabla_{1}\nabla_{2}\phi_{12}}_{L^{2}}^{2}+
			\n{\nabla_{2}\phi_{12}}_{L^{2}}^{2}}. \notag
	\end{align}
	As the term $III$ can be estimated in the same way as $II$, we also have
	\begin{align}\label{equ:III,estimate}
		III	\gtrsim & -\frac{\hbar^{2}}{N}\lrs{\n{\nabla_{1}\nabla_{2}\phi_{12}}_{L^{2}}^{2}+
			\n{\nabla_{2}\phi_{12}}_{L^{2}}^{2}}.
	\end{align}
	
	Together with estimates \eqref{equ:T1,T2,reduced form}--\eqref{equ:III,estimate}, we arrive at
	\begin{align*}
		\lra{T_{1}\psi,T_{2}\psi}\geq &
		\frac{1}{2}\int\lrc{|\phi_{12}|^{2}+\frac{\hbar^{2}}{2}|\nabla_{2}\phi_{12}|^{2}+\frac{\hbar^{2}}{2}|\nabla_{1}\phi_{12}|^{2}+\frac{\hbar^{4}}{4}|\nabla_{1}\nabla_{2}\phi_{12}|^{2}}\\
		&-C(\frac{N^{\be}}{N}\hbar^{2}+\frac{N^{2\be}}{N^{2}})\lrs{\n{\nabla_{1}\nabla_{2}\phi_{12}}_{L^{2}}^{2}+
			\n{\nabla_{1}\phi_{12}}_{L^{2}}^{2}}\\
		\geq& \frac{1}{16}\int |\phi_{12}|^{2}+\hbar^{2}|\nabla_{2}\phi_{12}|^{2}+\hbar^{2}|\nabla_{1}\phi_{12}|^{2}+\hbar^{4}|\nabla_{1}\nabla_{2}\phi_{12}|^{2}\\
=& \frac{1}{16}\lra{(1-\hbar^{2}\Delta_{1})(1-\hbar^{2}\Delta_{2})\phi_{12},\phi_{12}},
	\end{align*}
	where in the second-to-last inequality we have used that $N^{\be-1}\hbar^{-2}\ll 1$. With \eqref{equ:H2 energy estimate,T1T2}, we complete the proof of the estimate \eqref{equ:H2 energy estimate,lower bound}.
\end{proof}

\begin{proposition}\label{lemma:H2 energy estimate,psi}
Let $\be\in(0,1)$ and $N^{\be-1}\hbar^{2}\ll 1$. Define
\begin{align*}
\phi_{N,\hbar,12}(t,X_{N})=\frac{\psi_{N,\hbar}(t,X_{N})}{1-w_{N,\hbar}(x_{1}-x_{2})}.
\end{align*}
There exists a constant $C>0$ such that
\begin{align}\label{equ:H2 energy estimate,phi12}
\lra{(1-\hbar^{2}\Delta_{x_{1}})(1-\hbar^{2}\Delta_{x_{2}})\phi_{N,\hbar,12}(t),\phi_{N,\hbar,12}(t)}\leq C
\end{align}
for all $t\in \R$.
\end{proposition}
\begin{proof}
By the $(H_{N,\hbar})^{2}$ energy estimate in Lemma \ref{lemma:H2 energy estimate,lower bound}, we have
\begin{align*}
&\lra{(1-\hbar^{2}\Delta_{x_{1}})(1-\hbar^{2}\Delta_{x_{2}})\phi_{N,\hbar,12}(t),\phi_{N,\hbar,12}(t)}\\
\leq &\frac{16}{N(N-1)}\lra{\psi_{N,\hbar}(t),(N+H_{N,\hbar})^{2}\psi_{N,\hbar}(t)}\\
=&\frac{16}{N(N-1)}\lra{\psi_{N,\hbar}(0),(N+H_{N,\hbar})^{2}\psi_{N,\hbar}(0)}\\
\leq & 16(E_{0})^{2},
\end{align*}
where we have used the conservation of $(H_{N,\hbar})^{2}$ in the second-to-last equality, and the initial energy condition \eqref{equ:n-body energy bound,initial data condition} in the last inequality.
\end{proof}

\section{Functional Inequalities}\label{section:Functional Inequalities}

In the section, with the a-priori energy bound established in Proposition \ref{lemma:H2 energy estimate,psi}, we control the $\delta$-type potential parts $\mathcal{F}_{\delta}(t)$ and $\wt{\mathcal{F}}_{\delta}(t)$ and establish the functional inequalities \eqref{equ:functional inequality,fdelta,outline} and \eqref{equ:lower bound,fdelta,outline}. In Section \ref{section:Error Analysis of Two-Body Term}, we deal with the error analysis of the two-body term of $\wt{\mathcal{F}}_{\delta}(t)$ and then find the main part of $\wt{\mathcal{F}}_{\delta}(t)$. In Section \ref{section:Tamed Singularities}, we estimate the main part. By a replacement argument, we find proper approximations of $\mathcal{F}_{\delta}(t)$ and $\wt{\mathcal{F}}_{\delta}(t)$, and hence arrive at a reduction of the functional inequality. We then complete the proof of the reduced version of functional inequalities in Section \ref{section:Reduced Version of Functional Inequality}.

The main goal of the section is the following proposition which is the precise form of \eqref{equ:functional inequality,fdelta,outline} and \eqref{equ:lower bound,fdelta,outline}.

\begin{proposition}\label{lemma:functional inequality,fdelta,V}
Let $\be\in (0,1)$, we have the estimate
\begin{equation}\label{equ:functional inequality,fdelta}
\wt{\mathcal{F}}_{\delta}(t) \lesssim  \mathcal{F}_{\delta}(t)
+O(N^{\be-1}\hbar^{-6}+N^{-\frac{\be}{3}}\hbar^{-4}+N^{-\frac{1}{10}}\hbar^{-4})
\end{equation}
and a lower bound of $\mathcal{F}_{\delta}(t)$
\begin{align}\label{equ:lower bound,fdelta}
0 \leq \mathcal{F}_{\delta}(t) +O(N^{\be-1}\hbar^{-6} +N^{-\frac{\beta}{3}}\hbar^{-4}+N^{-\frac{1}{10}}\hbar^{-4}).
\end{align}
Here, the notation $O(a+b)$ is a shorthand for $O(a)+O(b)$
and the notation
$O(N^{-\al_{1}}\hbar^{-\al_{2}})$ denotes the same order of $N^{-\al_{1}}\hbar^{-\al_{2}}$ up to an unimportant constant\footnote{The constant could depend on the usual Sobolev constants and the fixed parameters such as the time $T_{0}$, the energy bound $E_{0}$, and the Sobolev norms of $(\rho,u)$ but the constant is independent of $(N,\hbar)$. }$C$.

\end{proposition}
\begin{proof}
We postpone the proof of Proposition \ref{lemma:functional inequality,fdelta,V} to the end of the Section \ref{section:Reduced Version of Functional Inequality}.
\end{proof}

\subsection{Error Analysis of Two-Body Term}\label{section:Error Analysis of Two-Body Term}

From the expression of $\wt{\mathcal{F}}_{\delta}(t)$
 \begin{align}
\wt{\mathcal{F}}_{\delta}(t)=&
\frac{N-1}{N}\int (u(t,x_{1})-u(t,x_{2}))\nabla V_N(x_{1}-x_{2})\rho_{N,\hbar}^{(2)}(t,x_{1},x_{2})dx_{1}dx_{2}\\
&-b_{0}\int \operatorname{div} u(t,x_{1})\rho(t,x_{1})\lrc{\rho(t,x_{1})-2\rho_{N,\hbar}^{(1)}(t,x_{1})}dx_{1},\notag
\end{align}
the difficult part is the two-body term
\begin{align*}
\int (u(t,x)-u(t,y))\cdot \nabla V_{N}(x-y) \rho_{N,\hbar}^{(2)}(t,x,y)dxdy.
\end{align*}
At first sight, the lack of a uniform regularity estimate for the two-body density function $\rho_{N,\hbar}^{(2)}(x,y)$ makes further analysis difficult.
With the singular correlation structure in mind, we decompose the two-body density function into the singular and regular parts
\begin{align*}
\rho_{N,\hbar}^{(2)}(x,y)=(1-w_{N,\hbar}(x-y))^{2} \frac{\rho_{N,\hbar}^{(2)}(x,y)}{(1-w_{N,\hbar}(x-y))^{2}},
\end{align*}
and rewrite the two-body term as
\begin{align}\label{equ:two-body term,rewritten form}
\int (u(x)-u(y))\cdot \nabla V_{N}(x-y) (1-w_{N,\hbar}(x-y))^{2} \frac{\rho_{N,\hbar}^{(2)}(x,y)}{(1-w_{N,\hbar}(x-y))^{2}}dxdy.
\end{align}
That is, the singularities come from the potential $\nabla V_{N}$ and the singular correlation function $(1-w_{N,\hbar}(x-y))^{2}$. As mentioned in \eqref{equ:cancellation structure,delta potential,outline} at the outline, a key observation to beat the singularities is
a cancellation structure from the difference coupled with the $\delta$-type potential
\begin{align}\label{equ:cancellation structure,delta potential}
(u(x)-u(y))V_{N}(x-y),
\end{align}
which would vanish as $N$ tends to the infinity. Such a structure is special for the $\delta$-type potential. Many common potentials including the Coulomb do not carry such a property.

We will prove that, based on \eqref{equ:H2 energy estimate,phi12}, the cancellation structure \eqref{equ:cancellation structure,delta potential} dominates the singularities generated by the delta-potential and singular correlation function, which allows us extract the main term from the two-body term.

\begin{lemma}\label{lemma:error analysis,main term}
Let $\be\in(0,1)$, we have
\begin{align}\label{equ:error analysis,main term}
&\int (u(x)-u(y))\cdot \nabla V_{N}(x-y) \rho_{N,\hbar}^{(2)}(x,y)dxdy\\
=&-\int \operatorname{div} u(x) V_{N}(x-y)\rho_{N,\hbar}^{(2)}(x,y)dxdy\pm O(N^{\be-1}\hbar^{-6}+N^{-\frac{\be}{2}}\hbar^{-4}),\notag
\end{align}
where the notation
$f=g\pm O(N,\hbar)$ means
$|f-g|\leq O(N,\hbar)$.
\end{lemma}
\begin{proof}
First, due to the singular correlation structure, we rewrite the two-body term as \eqref{equ:two-body term,rewritten form}.
To employ the cancellation structure \eqref{equ:cancellation structure,delta potential}, we take the derivative off $V_{N}$ by integrating by parts
\begin{align}
&\int (u(x)-u(y))\cdot \nabla V_{N}(x-y) (1-w_{N,\hbar}(x-y))^{2} \frac{\rho_{N,\hbar}^{(2)}(x,y)}{(1-w_{N,\hbar}(x-y))^{2}}dxdy\notag\\
=&-\int \operatorname{div}u(x) V_{N}(x-y) \rho_{N,\hbar}^{(2)}(x,y)dxdy\notag\\
&-\int (u(x)-u(y))V_{N}(x-y)\nabla_{x}\lrc{(1-w_{N,\hbar}(x-y))^{2} \frac{\rho_{N,\hbar}^{(2)}(x,y)}{(1-w_{N,\hbar}(x-y))^{2}}}dxdy.\label{equ:error term,x-y,F}
\end{align}
It remains to show the term \eqref{equ:error term,x-y,F} is indeed an error term. For simplicity, we set
$$w_{12}=w_{N,\hbar}(x_{1}-x_{2}),\quad \nabla w_{12}=(\nabla w_{N,\hbar})(x_{1}-x_{2}),\quad \phi_{N,\hbar,12}=(1-w_{12})\psi_{N,\hbar}.$$
Then we have
\begin{align}\label{equ:error analysis,error term}
&\int (u(x)-u(y))V_{N}(x-y)\nabla_{x}\lrc{(1-w_{N,\hbar}(x-y))^{2} \frac{\rho_{N,\hbar}^{(2)}(x,y)}{(1-w_{N,\hbar}(x-y))^{2}}}dxdy\\
=&\int (u(x_{1})-u(x_{2}))V_{N}(x_{1}-x_{2})\nabla_{x_{1}}\lrc{(1-w_{12})\phi_{N,\hbar,12}}^{2} dX_{N}\notag\\
 \leq& 2 \n{\nabla u}_{L^{\infty}}\int|x_{1}-x_{2}|V_{N}(x_{1}-x_{2})|\nabla w_{12}|(1-w_{12})|\phi_{N,\hbar,12}|^{2}
dX_{N}\notag\\
 &+2 \n{\nabla u}_{L^{\infty}}\int |x_{1}-x_{2}|V_{N}(x_{1}-x_{2})(1-w_{12})^{2}|\nabla_{x_{1}}\phi_{N,\hbar,12}|
 |\phi_{N,\hbar,12}|dX_{N} \notag\\
=:&2(A+B).\notag
\end{align}
We bound $A$ and $B$, using the properties of the scattering function, the two-body energy estimate \eqref{equ:H2 energy estimate,phi12} and the operator inequalities in Lemma \ref{lemma:Sobolev type estimate,3d}.

For the term $A$, by Lemma \ref{lemma:scattering equation}, we have the upper bound estimate
$$(1-w_{12})\leq 1,\quad
|\nabla w_{12}|\lesssim \frac{N^{2\be}}{N\hbar^{2}}.$$
Therefore, we arrive at
\begin{align}\label{equ:error analysis,error term A,estimate}
A=&\int \n{\nabla u}_{L^{\infty}}|x_{1}-x_{2}|V_{N}(x_{1}-x_{2})|\nabla w_{12}|(1-w_{12})|\phi_{N,\hbar,12}|^{2}
dX_{N}\\
\lesssim & \frac{N^{\be}}{N\hbar^{2}}\n{\nabla u}_{L^{\infty}}\int N^{\be}|x_{1}-x_{2}|V_{N}(x_{1}-x_{2}) |\phi_{N,\hbar,12}|^{2}
dX_{N}\notag\\
\lesssim& \frac{N^{\be}}{N\hbar^{2}}\n{\nabla u}_{L^{\infty}} \n{|x|V(x)}_{L^{1}}\lra{(1-\Delta_{x_{1}})(1-\Delta_{x_{2}})\phi_{N,\hbar,12},\phi_{N,\hbar,12}},\notag
\end{align}
where in the last inequality we used the operator inequality \eqref{equ:Sobolev type estimate,3d,L1}.

For the term $B$, we first discard $(1-w_{12})^{2}$ and then use Cauchy-Schwarz to get
\begin{align}\label{equ:error analysis,term B}
B=&\n{\nabla u}_{L^{\infty}}\int |x_{1}-x_{2}|V_{N}(x_{1}-x_{2})(1-w_{12})^{2}|\nabla_{x_{1}}\phi_{N,\hbar,12}|
 |\phi_{N,\hbar,12}|dX_{N}\notag\\
\leq &\frac{\n{\nabla u}_{L^{\infty}}}{N^{\be}}
\left[\al \lra{N^{\be}|x_{1}-x_{2}|V_{N}(x_{1}-x_{2})\phi_{N,\hbar,12},\phi_{N,\hbar,12}}\right.\\
&\left.+\al^{-1}\lra{N^{\be}|x_{1}-x_{2}|V_{N}(x_{1}-x_{2})\nabla_{x_{1}}\phi_{N,\hbar,12},\nabla_{x_{1}}\phi_{N,\hbar,12}}\right].\notag
\end{align}
By applying the operator inequality \eqref{equ:Sobolev type estimate,3d,L1} to the first term and the operator inequality \eqref{equ:Sobolev type estimate,3d,L2} to the second term on the r.h.s of \eqref{equ:error analysis,term B}, we obtain
\begin{align}\label{equ:error analysis,error term A,estimate}
B\leq&\frac{\n{\nabla u}_{L^{\infty}}}{N^{\be}}
\left(\al \n{|x|V(x)}_{L^{1}}
\lra{(1-\Delta_{x_{1}})(1-\Delta_{x_{2}})\phi_{N,\hbar,12},\phi_{N,\hbar,12}}\right.\\
&\left.+\al^{-1}
N^{\beta}\n{|x|V(x)}_{L^{\frac{3}{2}}}\lra{(1-\Delta_{x_{1}})(1-\Delta_{x_{2}})\phi_{N,\hbar,12},\phi_{N,\hbar,12}}\right)\notag\\
\lesssim & N^{-\frac{\be}{2}}\lra{(1-\Delta_{x_{1}})(1-\Delta_{x_{2}})\phi_{N,\hbar,12},\phi_{N,\hbar,12}},\notag
\end{align}
where in the last inequality we optimized the choice of $\al$.

Together with \eqref{equ:error analysis,error term} and estimates for the terms $A$ and $B$, we reach
\begin{align*}
&\int (u(x)-u(y))V_{N}(x-y)\nabla_{x}\lrc{(1-w_{N,\hbar}(x-y))^{2} \frac{\rho_{N,\hbar}^{(2)}(x,y)}{(1-w_{N,\hbar}(x-y))^{2}}}dxdy\\
\lesssim& \lrs{N^{\beta-1}\hbar^{-2}+N^{-\frac{\be}{2}}}
\lra{(1-\Delta_{x_{1}})(1-\Delta_{x_{2}})\phi_{N,\hbar,12},\phi_{N,\hbar,12}}\\
\lesssim& \lrs{N^{\be-1}\hbar^{-2}+N^{-\frac{\be}{2}}}\hbar^{-4}
\lra{(1-\hbar^{2}\Delta_{x_{1}})(1-\hbar^{2}\Delta_{x_{2}})\phi_{N,\hbar,12},\phi_{N,\hbar,12}}\\
\lesssim& N^{\be-1}\hbar^{-6}+N^{-\frac{\be}{2}}\hbar^{-4},
\end{align*}
where in the last inequality we used the two-body $H^{1}$ energy bound \eqref{equ:H2 energy estimate,phi12}. This completes the proof of \eqref{equ:error analysis,main term}.

\end{proof}

\subsection{Tamed Singularities}\label{section:Tamed Singularities}
As a result of the error analysis of the two-body term, we are able to capture
the main term
\begin{align*}
\int \operatorname{div} u(x) V_{N}(x-y)\rho_{N,\hbar}^{(2)}(x,y)dxdy.
\end{align*}
Using the identity approximation to the one-body term of $\wt{\mathcal{F}}_{\delta}(t)$, we arrive at
\begin{align*}
\wt{\mathcal{F}}_{\delta}(t)\sim & -\int \operatorname{div} u(t,x) V_{N}(x-y)
 \lrc{\frac{N-1}{N}\rho_{N,\hbar}^{(2)}(t,x,y)\right. \\
& \left. \quad \quad   -\rho_{N,\hbar}^{(1)}(t,x)\rho(t,y)-\rho(t,x)\rho_{N,\hbar}^{(1)}(t,y)+\rho(t,x)\rho(t,y)}dxdy.
\end{align*}
By the identity approximation again, we also have the approximation of $\mathcal{F}_{\delta}(t)$ that
\begin{align*}
\mathcal{F}_{\delta}(t)\sim &\int V_{N}(x-y) \lrc{\frac{N-1}{N}\rho_{N,\hbar}^{(2)}(t,x,y)\right.\\
&\quad \quad \left.-\rho_{N,\hbar}^{(1)}(t,x)\rho(t,y)-\rho(t,x)\rho_{N,\hbar}^{(1)}(t,y)+\rho(t,x)\rho(t,y)}dxdy.\notag
\end{align*}
We now need to deal with the sharp singularity of $V_{N}(x)$.
We tame the singularity by replacing $V_{N}(x)$ with a slowly varying potential
with a number of good properties. However, the replacement relies on the regularity of the integrand. Therefore, we
again need to decompose the two-body density function as the singular and relatively regular parts.
We obtain proper approximations of $\mathcal{F}_{\delta}(t)$ and $\wt{\mathcal{F}}_{\delta}(t)$ and arrive at a reduced version of functional inequalities via a careful analysis.

The following is the main lemma of the section.

\begin{lemma}\label{lemma:replacement argument}
Let
$$
G(x)=\lrs{\frac{1}{\pi}}^{\frac{3}{2}}e^{-|x|^{2}},\quad G_{N}(x)=N^{3\eta}G(N^{\eta}x).
$$
Then for the two-body term we have
\begin{align}\label{equ:replacement,F,two body}
&\int F(x) V_{N}(x-y)\rho_{N,\hbar}^{(2)}(x,y)dxdy\\
=&b_{0}\int  F(x)G_{N}(x-y)\rho_{N,\hbar}^{(2)}(x,y)dxdy\pm O(N^{\be-1}\hbar^{-6} +N^{-\frac{\beta}{3}}\hbar^{-4}+N^{-\frac{\eta}{3}}\hbar^{-4}),\notag
\end{align}
and for the one-body term we have
\begin{align}\label{equ:replacement,F,one body}
&b_{0}\int F(x)\rho(x)\lrc{\rho(x)-2\rho_{N,\hbar}^{(1)}(x)}dx\\
=&b_{0}\int F(x)G_{N}(x-y)\lrc{-\rho_{N,\hbar}^{(1)}(x)\rho(y)-\rho(x)\rho_{N,\hbar}^{(1)}(y)+\rho(x)\rho(y)}dxdy\pm O(N^{-\eta}).\notag
\end{align}
In particular, given $F(x)=1$, we have the approximation of $\mathcal{F}_{\delta}(t)$
\begin{align}\label{equ:replacement,F=1,Fdelta}
\mathcal{F}_{\delta}(t)=&b_{0}\int G_{N}(x-y) \lrc{\frac{N-1}{N}\rho_{N,\hbar}^{(2)}(x,y)\right.\\
&\quad \quad \left.-\rho_{N,\hbar}^{(1)}(x)\rho(y)-\rho(x)\rho_{N,\hbar}^{(1)}(y)+\rho(x)\rho(y)}dxdy\notag\\
&\pm O(N^{\be-1}\hbar^{-6} +N^{-\frac{\beta}{3}}\hbar^{-4}+N^{-\frac{\eta}{3}}\hbar^{-4}),\notag
\end{align}
and given $F(x)=\operatorname{div} u(x)$, we have the approximation of $\wt{\mathcal{F}}_{\delta}(t)$
\begin{align}\label{equ:replacement,F=div,Fdelta}
\wt{\mathcal{F}}_{\delta}(t)=&-b_{0}\int \operatorname{div} u(x) G_{N}(x-y)
 \lrc{\frac{N-1}{N}\rho_{N,\hbar}^{(2)}(x,y)\right.\\
& \left. \quad \quad   -\rho_{N,\hbar}^{(1)}(x)\rho(y)-\rho(x)\rho_{N,\hbar}^{(1)}(y)+\rho(x)\rho(y)}dxdy \notag\\
&\pm O(N^{\be-1}\hbar^{-6} +N^{-\frac{\beta}{3}}\hbar^{-4}+N^{-\frac{\eta}{3}}\hbar^{-4}).\notag
\end{align}
\end{lemma}
\begin{proof}
For \eqref{equ:replacement,F,two body}, we recall $\phi_{N,\hbar,12}=(1-w_{12})\psi_{N,\hbar}$ and rewrite
\begin{align*}
&\int F(x) V_{N}(x-y)\rho_{N,\hbar}^{(2)}(x,y)dxdy\\
=&\lra{F(x)V_{N}(x-y)(1-w_{12})^{2}\phi_{N,\hbar,12},\phi_{N,\hbar,12}}\\
=&\lra{F(x)V_{N}(x-y)\phi_{N,\hbar,12},\phi_{N,\hbar,12}}+\lra{F(x)V_{N}(x-y)(-2w_{12}+(w_{12})^{2})\phi_{N,\hbar,12},\phi_{N,\hbar,12}}\\
=&I+II.
\end{align*}

For the term $I$, we use the Poincar$\acute{e}$ type inequality with $\theta=\frac{1}{3}$ in Lemma \ref{lemma:poincare type inequality,appendix} to obtain
\begin{align*}
&|\lra{F(x)(V_{N}(x-y)-b_{0}\delta(x-y))\phi_{N,\hbar,12},\phi_{N,\hbar,12}}|\\
\lesssim &N^{-\frac{\beta}{3}}\n{\lra{\nabla_{x_{1}}}\lra{\nabla_{x_{2}}}F(x_{1})\phi_{N,\hbar,12}}_{L^{2}}
\n{\lra{\nabla_{x_{1}}}\lra{\nabla_{x_{2}}}\phi_{N,\hbar,12}}_{L^{2}}\\
\lesssim& N^{-\frac{\beta}{3}}(\n{F}_{L^{\infty}}+\n{\nabla F}_{L^{\infty}})\n{\lra{\nabla_{x_{1}}}\lra{\nabla_{x_{2}}}\phi_{N,\hbar,12}}_{L^{2}}^{2}
\lesssim N^{-\frac{\beta}{3}}\hbar^{-4},
\end{align*}
where in the last inequality we used the two-body energy bound \eqref{equ:H2 energy estimate,phi12}.

For the term $II$, by Lemma \ref{lemma:scattering equation}, we have
$|w_{12}|\lesssim N^{\be-1}\hbar^{-2}$.
Therefore, we get
\begin{align*}
II\lesssim &\n{F}_{L^{\infty}}\lra{V_{N}(x-y)|w_{12}|\phi_{N,\hbar,12},\phi_{N,\hbar,12}}\\
\lesssim & N^{\be-1}\hbar^{-2}\lra{V_{N}(x-y)\phi_{N,\hbar,12},\phi_{N,\hbar,12}}\\
\lesssim & N^{\be-1}\hbar^{-2} \lra{(1-\Delta_{x_{1}})(1-\Delta_{x_{2}})\phi_{N,\hbar,12},\phi_{N,\hbar,12}}
\lesssim  N^{\be-1}\hbar^{-6},
\end{align*}
where we used the operator inequality \eqref{equ:Sobolev type estimate,3d,L1} in the second-to-last inequality and the two-body energy bound \eqref{equ:H2 energy estimate,phi12} in the last inequality.

In the same way, we also obtain
\begin{align*}
&b_{0}\int F(x) G_{N}(x-y)\rho_{N,\hbar}^{(2)}(x,y)dxdy\\
=&b_{0}\int F(x) \delta(x-y)\rho_{N,\hbar}^{(2)}(x,y)dxdy\pm O(N^{-\frac{\eta}{3}}\hbar^{-4}+N^{\be-1}\hbar^{-6}).
\end{align*}
Then by the triangle inequality, we arrive at
\begin{align*}
&\bbabs{\int F(x) V_{N}(x-y)\rho_{N,\hbar}^{(2)}(x,y)dxdy-b_{0}\int F(x) G_{N}(x-y)\rho_{N,\hbar}^{(2)}(x,y)dxdy}\\
\lesssim& N^{\be-1}\hbar^{-6} +N^{-\frac{\beta}{3}}\hbar^{-4}+N^{-\frac{\eta}{3}}\hbar^{-4},
\end{align*}
which completes the proof of \eqref{equ:replacement,F,two body}.

For \eqref{equ:replacement,F,one body}, we rewrite
\begin{align}
&\int F(x)\rho(x)\lrc{\rho(x)-2\rho_{N,\hbar}^{(1)}(x)}dx\notag\\
=&\int F(x)\delta(x-y)\lrc{-\rho_{N,\hbar}^{(1)}(x)\rho(y)-\rho(x)\rho_{N,\hbar}^{(1)}(y)+\rho(x)\rho(y)}dxdy\notag\\
=&\int F(x)G_{N}(x-y)\lrc{-\rho_{N,\hbar}^{(1)}(x)\rho(y)-\rho(x)\rho_{N,\hbar}^{(1)}(y)+\rho(x)\rho(y)}dxdy\notag\\
&-\lra{F\rho_{N,\hbar}^{(1)},(G_{N}-\delta)*\rho}-\lra{(G_{N}-\delta)*(F\rho),\rho_{N,\hbar}^{(1)}}+
\lra{F\rho,(G_{N}-\delta)*\rho}.\label{equ:replacement,F,one body,error term}
\end{align}
For the error terms in \eqref{equ:replacement,F,one body,error term}, we use H\"{o}lder and Sobolev inequalities
to get
\begin{align*}
&|\lra{F\rho_{N,\hbar}^{(1)},(G_{N}-\delta)*\rho}|+|\lra{(G_{N}-\delta)*(F\rho),\rho_{N,\hbar}^{(1)}}|+
|\lra{F\rho,(G_{N}-\delta)*\rho}|\\
\leq& \n{F}_{L^{\infty}}(\n{\rho_{N,\hbar}^{(1)}}_{L^{1}}+\n{\rho}_{L^{1}})\n{(G_{N}-\delta)*\rho}_{L^{\infty}}+
\n{\rho_{N,\hbar}^{(1)}}_{L^{1}}\n{(G_{N}-\delta)*(F\rho)}_{L^{\infty}}\\
\lesssim&\n{F}_{L^{\infty}}(\n{\rho_{N,\hbar}^{(1)}}_{L^{1}}+\n{\rho}_{L^{1}})\n{(G_{N}-\delta)*\lra{\nabla}^{2}\rho}_{L^{2}}+
\n{\rho_{N,\hbar}^{(1)}}_{L^{1}}\n{(G_{N}-\delta)*\lra{\nabla}^{2}(F\rho)}_{L^{2}}\\
\lesssim& N^{-\eta}\lrs{\n{F}_{L^{\infty}}\n{\rho}_{H^{3}}+\n{F\rho}_{H^{3}}}\\
\lesssim& N^{-\eta}\n{F}_{H^{3}}\n{\rho}_{H^{3}},
\end{align*}
where in the second-to-last inequality we used Lemma \ref{lemma:quantative estimate for identity approximation} and the mass conservation, and in the last inequality we used Leibniz rule and Sobolev inequality. Therefore, we complete the proof of \eqref{equ:replacement,F,one body}.

For \eqref{equ:replacement,F=1,Fdelta}, by taking $F(x)=1$ in \eqref{equ:replacement,F,two body} and \eqref{equ:replacement,F,one body}, we arrive at the approximation of $\mathcal{F}_{\delta}(t)$.

For \eqref{equ:replacement,F=div,Fdelta}, by the error analysis \eqref{equ:error analysis,main term} in Lemma \ref{lemma:error analysis,main term} we get
\begin{align*}
\wt{\mathcal{F}}_{\delta}=&
-b_{0}\int \operatorname{div} u(x) V_{N}(x-y)\rho_{N,\hbar}^{(2)}(x,y)dxdy\\
&-b_{0}\int \operatorname{div} u(x)\rho(x)\lrc{\rho(x)-2\rho_{N,\hbar}^{(1)}(x)}dx\\
&\pm O(N^{\be-1}\hbar^{-6}+N^{-\frac{\be}{2}}\hbar^{-4}).
\end{align*}

Then by taking $F(x)=\operatorname{div} u(x)$ in \eqref{equ:replacement,F,two body} and \eqref{equ:replacement,F,one body}, we get the approximation \eqref{equ:replacement,F=div,Fdelta} of $\wt{\mathcal{F}}_{\delta}(t)$.
\end{proof}

\subsection{Reduced Version of Functional Inequality}\label{section:Reduced Version of Functional Inequality}
After the analysis of error terms and simplification, we now work with a reduced form of functional inequality
\begin{align}
 &\int \operatorname{div} u(x) G_{N}(x-y)
 \lrc{\frac{N-1}{N}\rho_{N,\hbar}^{(2)}(x,y)-\rho_{N,\hbar}^{(1)}(x)\rho(y)-\rho(x)\rho_{N,\hbar}^{(1)}(y)
 +\rho(x)\rho(y)}dxdy\label{equ:reduced functional inequality}\\
 \lesssim& \int G_{N}(x-y) \lrc{\frac{N-1}{N}\rho_{N,\hbar}^{(2)}(x,y)-\rho_{N,\hbar}^{(1)}(x)\rho(y)-\rho(x)\rho_{N,\hbar}^{(1)}(y)+\rho(x)\rho(y)}dxdy +o(1),\notag
\end{align}
which is more concise than the original functional inequality.
However, it is unknown whether or not the integrand
\begin{align}
\frac{N-1}{N}\rho_{N,\hbar}^{(2)}(x,y)-\rho_{N,\hbar}^{(1)}(x)\rho(y)-\rho(x)\rho_{N,\hbar}^{(1)}(y)+\rho(x)\rho(y)
\end{align}
is non-negative, so we cannot directly bound the term $\operatorname{div} u(x)$ in \eqref{equ:reduced functional inequality}. We prove that, if integrated against $G_{N}(x-y)$, (4.18) provides a non-negative contribution up to a small correction and use that to prove the lower bound of $\mathcal{F}_{\delta}(t)$.
The special structure of a relatively slowly varying and explicit potential $G_{N}(x)$ plays a critical role in establishing the reduced version of functional inequality.
We then complete the proof of Proposition \ref{lemma:functional inequality,fdelta,V}.

\begin{lemma}[Reduced Version of Functional Inequality]\label{lemma:functional inequality,fdelta}
Let $\eta<\frac{1}{3}$ to be determined, we have
\begin{align}\label{equ:functional inequality,fdelta,G}
 &\int F(x) G_{N}(x-y)
 \lrc{\frac{N-1}{N}\rho_{N,\hbar}^{(2)}(x,y)-\rho_{N,\hbar}^{(1)}(x)\rho(y)-\rho(x)\rho_{N,\hbar}^{(1)}(y)+\rho(x)\rho(y)}dxdy\\
 \leq& \n{F}_{L^{\infty}}\int G_{N}(x-y) \lrc{\frac{N-1}{N}\rho_{N,\hbar}^{(2)}(x,y)-\rho_{N,\hbar}^{(1)}(x)\rho(y)-\rho(x)\rho_{N,\hbar}^{(1)}(y)+\rho(x)\rho(y)}dxdy \notag \\
 &+O(N^{-\eta}\hbar^{-2}+N^{3\eta-1}).\notag
\end{align}
\end{lemma}
\begin{proof}
For simplicity, set $\rho_{N,\hbar}(X_{N})=|\psi_{N,\hbar}(X_{N})|^{2}$.
By the symmetry of  $\rho_{N,\hbar}(X_{N})$, we can write
\begin{align*}
& \int F(x) G_{N}(x-y)\lrc{\frac{N-1}{N}\rho_{N,\hbar}^{(2)}(x,y)-\rho_{N,\hbar}^{(1)}(x)\rho(y)-\rho(x)\rho_{N,\hbar}^{(1)}(y)+\rho(x)\rho(y)}dxdy\\
=&\frac{1}{N^{2}}\sum_{i\neq j}^{N}\int F(x_{i})G_{N}(x_{i}-x_{j})\rho_{N,\hbar}(X_{N})dX_{N}+\int F(x)G_{N}(x-y)\rho(x)\rho(y)dxdy\\
&-\frac{1}{N}\sum_{i=1}^{N}\int \int F(x_{i})G_{N}(x_{i}-y)\rho(y)dy \rho_{N,\hbar}(X_{N})dX_{N}\\
&-\frac{1}{N}\sum_{j=1}^{N}\int \int F(x)G_{N}(x-x_{j})\rho(x)dx \rho_{N,\hbar}(X_{N})dX_{N}\\
=&\int F(x)G_{N}(x-y)\lrc{\frac{1}{N^{2}}\sum_{i\neq j}^{N}\delta_{x_{i}}(x)\delta_{x_{j}}(y)+\rho(x)\rho(y)\right.\\
&\left.-\frac{1}{N}\sum_{i=1}^{N}\delta_{x_{i}}(x)\rho(y)-\rho(x)\frac{1}{N}\sum_{j=1}^{N}\delta_{x_{j}}(y)}dxdy \rho_{N,\hbar}(X_{N})dX_{N}.
\end{align*}
To simplify, we define the measure
\begin{align}
\nu_{X_{N}}(dx)=\frac{1}{N}\sum_{i=1}^{N}\delta_{x_{i}}(dx)-\rho(x)dx.
\end{align}
We rewrite
\begin{align}\label{equ:functional inequality,div,G term,rewritten form}
& \int F(x) G_{N}(x-y)\lrc{\frac{N-1}{N}\rho_{N,\hbar}^{(2)}(x,y)-\rho_{N,\hbar}^{(1)}(x)\rho(y)-\rho(x)\rho_{N,\hbar}^{(1)}(y)+\rho(x)\rho(y)}dxdy\\
=&\int F(x)G_{N}(x-y)\nu_{X_{N}}(dx)\nu_{X_{N}}(dy)\rho_{N,\hbar}(X_{N})dX_{N}-\frac{G_{N}(0)}{N}\int F(x)\rho_{N,\hbar}^{(1)}(x)dx.\notag
\end{align}
where the last term on the r.h.s of \eqref{equ:functional inequality,div,G term,rewritten form} comes from the diagonal summation.
In particular, if we take $F(x)=1$, we also have
\begin{align}\label{equ:functional inequality,G term,rewritten form}
& \int G_{N}(x-y)\lrc{\frac{N-1}{N}\rho_{N,\hbar}^{(2)}(x,y)-\rho_{N,\hbar}^{(1)}(x)\rho(y)-\rho(x)\rho_{N,\hbar}^{(1)}(y)+\rho(x)\rho(y)}dxdy\\
=&\int G_{N}(x-y)\nu_{X_{N}}(dx)\nu_{X_{N}}(dy)\rho_{N,\hbar}(X_{N})dX_{N}-\frac{G_{N}(0)}{N}\int \rho_{N,\hbar}^{(1)}(x)dx.\notag
\end{align}
Note that
\begin{align}\label{equ:equ:functional inequality,diagonal term}
\frac{G_{N}(0)}{N}\int F(x)\rho_{N,\hbar}^{(1)}(x)dx\leq N^{3\eta-1}\n{F}_{L^{\infty}},
\end{align}
which is a smallness term as long as $\eta<\frac{1}{3}$.

Next, we get into the analysis of the main term.
Note that the convolution property of the Gaussian function $G$, which is
\begin{align}\label{equ:G,square root}
G_{N}(x-y)=\int G_{0,N}(x-z)G_{0,N}(z-y)dz,
\end{align}
where $G_{0,N}(x)=N^{3\eta}G_{0}(N^{\eta}x)$ and $G_{0}(x)=\lrs{\frac{2}{\pi}}^{\frac{3}{2}}e^{-2|x|^{2}}$. Putting \eqref{equ:G,square root} into the main term of \eqref{equ:functional inequality,G term,rewritten form} gives
\begin{align}\label{equ:functional inequality,div,G term,main term}
&\int F(x)G_{N}(x-y)\nu_{X_{N}}(dx)\nu_{X_{N}}(dy)\rho_{N,\hbar}(X_{N})dX_{N}\\
=&\int F(x)G_{0,N}(x-z)G_{0,N}(z-y)\nu_{X_{N}}(dx)\nu_{X_{N}}(dy) \rho_{N,\hbar}(X_{N})dz dX_{N}\notag\\
=&A+B,\notag
\end{align}
where
\begin{align}
A=&\int (F(x)-F(z))G_{0,N}(x-z)G_{0,N}(z-y)\nu_{X_{N}}(dx)\nu_{X_{N}}(dy)\rho_{N,\hbar}(X_{N})dzdX_{N},\\
B=&\int F(z)G_{0,N}(x-z)G_{0,N}(z-y)\nu_{X_{N}}(dx)\nu_{X_{N}}(dy) \rho_{N,\hbar}(X_{N})dzdX_{N}.
\end{align}
Thus, we are left to bound the terms $A$ and $B$.

For the term $A$, we use Cauchy-Schwarz inequality to get
\begin{align*}
A^{2}\leq& \int \lrc{\int (F(x)-F(z))G_{0,N}(x-z)\nu_{X_{N}}(dx)}^{2} \rho_{N,\hbar}(X_{N}) dzdX_{N}  \\
&\cdot\int \lrc{\int G_{0,N}(z-y)\nu_{X_{N}}(dy)}^{2} \rho_{N,\hbar}(X_{N}) dzdX_{N} \\
\leq& 2(A_{1}+A_{2})\int \lrc{\int G_{0,N}(z-y)\nu_{X_{N}}(dy)}^{2} \rho_{N,\hbar}(X_{N}) dzdX_{N},
\end{align*}
where
\begin{align*}
A_{1}=&\int \lrc{\int (F(x)-F(z))G_{0,N}(x-z)\frac{1}{N}\sum_{i=1}^{N}\delta_{x_{i}}(dx) }^{2} \rho_{N,\hbar}(X_{N}) dzdX_{N},\\
A_{2}=&\int \lrc{\int (F(x)-F(z))G_{0,N}(x-z)\rho(x)dx}^{2} \rho_{N,\hbar}(X_{N}) dzdX_{N}.
\end{align*}

For $A_{1}$, we further decompose it into two parts
$A_{1}=A_{11}+A_{12}$,
where
the diagonal part is
\begin{align*}
A_{11}=&\frac{1}{N^{2}}\sum_{i=1}^{N}\int  (F(x_{i})-F(z))G_{0,N}(x_{i}-z)(F(x_{i})-F(z))G_{0,N}(x_{i}-z) \rho_{N,\hbar}(X_{N}) dzdX_{N},
\end{align*}
and the off-diagonal part is
\begin{align*}
A_{12}=&\frac{1}{N^{2}}\sum_{i\neq j}^{N}\int  (F(x_{i})-F(z))G_{0,N}(x_{i}-z)(F(x_{j})-F(z))G_{0,N}(x_{j}-z) \rho_{N,\hbar}(X_{N}) dzdX_{N}.
\end{align*}

For $A_{11}$, by the symmetry of $\rho_{N,\hbar}(X_{N})$, we have
\begin{align*}
A_{11}\leq& \frac{1}{N}\int  (F(x_{1})-F(z))G_{0,N}(x_{1}-z)(F(x_{1})-F(z))G_{0,N}(x_{1}-z) \rho_{N,\hbar}(X_{N}) dzdX_{N}\\
\leq &\frac{\n{\nabla F}_{L^{\infty}}^{2}}{N}\int \lrs{|x_{1}-z|G_{0,N}(x_{1}-z)}^{2}\rho_{N,\hbar}(X_{N}) dzdX_{N}\\
= &\frac{\n{\nabla F}_{L^{\infty}}^{2}}{N}\n{|x|G_{0,N}(x)}_{L^{2}}^{2}\int \rho_{N,\hbar}(X_{N})dX_{N}
\lesssim N^{\eta-1},
\end{align*}
where in the last inequality we used that $\n{|x|G_{0,N}(x)}_{L^{2}}^{2}\lesssim N^{\eta}$ and the mass conservation for $\rho_{N,\hbar}(X_{N})$.

For $A_{12}$, by the symmetry of $\rho_{N,\hbar}(X_{N})$, we also have
\begin{align*}
A_{12}\leq& \int  |(F(x_{1})-F(z))G_{0,N}(x_{1}-z)(F(x_{2})-F(z))G_{0,N}(x_{2}-z)| \rho_{N,\hbar}(X_{N}) dzdX_{N}\\
\leq & \n{\nabla F}_{L^{\infty}}^{2}\int |x_{1}-z|G_{0,N}(x_{1}-z)|x_{2}-z|G_{0,N}(x_{2}-z)| \rho_{N,\hbar}(X_{N}) dzdX_{N}\\
\leq& \frac{\n{\nabla F}_{L^{\infty}}^{2}}{N^{2\eta}}\int G_{1,N}(x_{1}-x_{2}) \rho_{N,\hbar}(X_{N}) dX_{N}\\
=& \frac{\n{\nabla F}_{L^{\infty}}^{2}}{N^{2\eta}} \lra{G_{1,N}(x_{1}-x_{2})\psi_{N,\hbar},\psi_{N,\hbar}},
\end{align*}
where
\begin{align*}
G_{1,N}(x_{1}-x_{2})=\int N^{\eta}|x_{1}-z|G_{0,N}(x_{1}-z)N^{\eta}|x_{2}-z|G_{0,N}(x_{2}-z)|dz.
\end{align*}
To bound $A_{12}$, we recall $\phi_{N,\hbar,12}=(1-w_{12})\psi_{N,\hbar}$ then get
\begin{align*}
A_{12}\lesssim& N^{-2\eta}\lra{G_{1,N}(x_{1}-x_{2})(1-w_{12})^{2}\phi_{N,\hbar,12},\phi_{N,\hbar,12}}\\
\lesssim&N^{-2\eta}\lra{G_{1,N}(x_{1}-x_{2})\phi_{N,\hbar,12},\phi_{N,\hbar,12}}\\
\lesssim& N^{-2\eta}\n{G_{1,N}}_{L^{1}}\lra{(1-\Delta_{1})(1-\Delta_{2})\phi_{N,\hbar,12},\phi_{N,\hbar,12}}
\lesssim N^{-2\eta}\hbar^{-4},
\end{align*}
where we discarded $(1-w_{12})^{2}$ in the second line and used the operator inequality \eqref{equ:Sobolev type estimate,3d,L1} in the second-to-last inequality, and the two-body $H^{1}$ energy bound \eqref{equ:H2 energy estimate,phi12} in the last inequality.

For $A_{2}$, we rewrite
\begin{align*}
A_{2}=& \n{F(G_{0,N}*\rho)-G_{0,N}*(F\rho)}_{L^{2}}^{2}\int \rho_{N,\hbar}(X_{N})dX_{N}\\
=&\n{F(G_{0,N}*\rho)-G_{0,N}*(F\rho)}_{L^{2}}^{2},
\end{align*}
where in the last inequality we used the mass conservation for $\rho_{N,\hbar}(X_{N})$. By the triangle, H\"{o}lder inequalities and Lemma \ref{lemma:quantative estimate for identity approximation} we get
\begin{align*}
A_{2}\leq&2\n{F(G_{0,N}*\rho)-F\rho}_{L^{2}}^{2}+2\n{F\rho-G_{0,N}*(F\rho)}_{L^{2}}^{2}\\
\lesssim&\n{F}_{L^{\infty}}^{2}\n{(G_{0,N}-\delta)*\rho}_{L^{2}}^{2}+\n{(G_{0,N}-\delta)*(F\rho)}_{L^{2}}^{2}\\
\lesssim&\frac{1}{N^{2\eta}}\n{F}_{L^{\infty}}^{2}\n{\lra{\nabla}\rho}_{L^{2}}^{2}+
\frac{1}{N^{2\eta}}\n{\lra{\nabla}(F\rho)}_{L^{2}}^{2}
\lesssim N^{-2\eta}.
\end{align*}

To sum up, we complete the estimates for the term $A$ and reach
\begin{align}\label{equ:functional inequality,div,A,finial}
A \leq \sqrt{A_{11}}+\sqrt{A_{12}}+\sqrt{A_{2}} \lesssim N^{\frac{\eta-1}{2}}+N^{-\eta}\hbar^{-2}\lesssim N^{-\eta}\hbar^{-2},
\end{align}
where in the last inequality we used that $N^{\frac{\eta-1}{2}}\leq N^{-\eta}$ for $\eta<\frac{1}{3}$.

For the term $B$, we rewrite
\begin{align*}
B=&\int F(z)G_{0,N}(x-z)G_{0,N}(z-y)\nu_{X_{N}}(dx)\nu_{X_{N}}(dy)dz \rho_{N,\hbar}(X_{N})dX_{N}\\
=&\int F(z)\abs{G_{0,N}*\nu_{X_{N}}(z)}^{2} \rho_{N,\hbar}(X_{N})dz dX_{N}.
\end{align*}
Observe that
\begin{align*}
\abs{G_{0,N}*\nu_{X_{N}}(z)}^{2} \rho_{N,\hbar}(X_{N})\geq 0
\end{align*}
for a.e. $(z,X_{N})\in \R^{3}\times \R^{3N}$.
Therefore, we can directly bound $F(z)$ and get
\begin{align}\label{equ:functional inequality,div,B,finial}
B\leq& \n{F}_{L^{\infty}}\int \abs{G_{0,N}*\nu_{X_{N}}(z)}^{2} \rho_{N,\hbar}(X_{N})dz dX_{N}\\
=&\n{F}_{L^{\infty}}\int  \int G_{0,N}(x-z)G_{0,N}(z-y)dz\nu_{X_{N}}(dx)\nu_{X_{N}}(dy) \rho_{N,\hbar}(X_{N})dX_{N}\notag\\
=&\n{F}_{L^{\infty}}\int  G_{N}(x-y)\nu_{X_{N}}(dx)\nu_{X_{N}}(dy) \rho_{N,\hbar}(X_{N})dX_{N}\notag\\
=&\n{F}_{L^{\infty}}\int G_{N}(x-y) \lrc{\frac{N-1}{N}\rho_{N,\hbar}^{(2)}(x,y)-\rho_{N,\hbar}^{(1)}(x)\rho(y)-\rho(x)\rho_{N,\hbar}^{(1)}(y)+\rho(x)\rho(y)}dxdy\notag\\
&+\n{F}_{L^{\infty}}\frac{G_{N}(0)}{N}\int \rho_{N,\hbar}^{(1)}(x)dx,\notag
\end{align}
where in the second-to-last equality we used the property \eqref{equ:G,square root}, and in the last equality we used the equation \eqref{equ:functional inequality,G term,rewritten form}.

With the approximation forms \eqref{equ:functional inequality,div,G term,rewritten form} and \eqref{equ:functional inequality,div,G term,main term}, we use estimates \eqref{equ:equ:functional inequality,diagonal term}, \eqref{equ:functional inequality,div,A,finial} for $A$ and \eqref{equ:functional inequality,div,B,finial} for $B$ to arrive at
\begin{align*}
&\int F(x) G_{N}(x-y)\lrc{\frac{N-1}{N}\rho_{N,\hbar}^{(2)}(x,y)-\rho_{N,\hbar}^{(1)}(x)\rho(y)-\rho(x)\rho_{N,\hbar}^{(1)}(y)+\rho(x)\rho(y)}dxdy\\
=&A+B+O(N^{3\eta-1})\\
\leq&\n{F}_{L^{\infty}}\int G_{N}(x-y) \lrc{\frac{N-1}{N}\rho_{N,\hbar}^{(2)}(x,y)-\rho_{N,\hbar}^{(1)}(x)\rho(y)-\rho(x)\rho_{N,\hbar}^{(1)}(y)+\rho(x)\rho(y)}dxdy\\
&+O(N^{-\eta}\hbar^{-2}+N^{3\eta-1}),
\end{align*}
which is the desired estimate \eqref{equ:functional inequality,fdelta,G}.
\end{proof}

To prove the lower bound estimate \eqref{equ:lower bound,fdelta} for $\mathcal{F}_{\delta}(t)$, we give the following estimate.
\begin{lemma}\label{lemma:two body,one body}
Let $\eta<\frac{1}{3}$ to be determined, we have
\begin{align}\label{equ:two body,one body}
&\int G_{N}(x-y) \lrc{\frac{N-1}{N}\rho_{N,\hbar}^{(2)}(x,y)-\rho_{N,\hbar}^{(1)}(x)\rho(y)-\rho(x)\rho_{N,\hbar}^{(1)}(y)+\rho(x)\rho(y)}dxdy\\
\geq &\int G_{N}(x-y) (\rho_{N,\hbar}^{(1)}(x)-\rho(x))(\rho_{N,\hbar}^{(1)}(y)-\rho(y))dxdy-O(N^{3\eta-1}).\notag
\end{align}
\end{lemma}
\begin{proof}
We decompose
\begin{align*}
&\int G_{N}(x-y) \lrc{\frac{N-1}{N}\rho_{N,\hbar}^{(2)}(x,y)-\rho_{N,\hbar}^{(1)}(x)\rho(y)-\rho(x)\rho_{N,\hbar}^{(1)}(y)+\rho(x)\rho(y)}dxdy\\
=&I+II,
\end{align*}
where
\begin{align}
I=&\int G_{N}(x-y) (\rho_{N,\hbar}^{(1)}(x)-\rho(x))(\rho_{N,\hbar}^{(1)}(y)-\rho(y))dxdy,\\
II=&\int G_{N}(x-y)\lrc{\frac{N-1}{N}\rho_{N,\hbar}^{(2)}(x,y)-\rho_{N,\hbar}^{(1)}(x)\rho_{N,\hbar}^{(1)}(y)}dxdy.
\end{align}
It suffices to prove a lower bound of the term $II$. By the symmetry of the density function $\rho_{N,\hbar}(X_{N})$, we rewrite
\begin{align*}
II=&\int G_{N}(x-y)\lrc{\frac{N-1}{N}\rho_{N,\hbar}^{(2)}(x,y)-\rho_{N,\hbar}^{(1)}(x)\rho_{N,\hbar}^{(1)}(y)}dxdy\\
=&\int G_{N}(x-y)\lrc{\frac{1}{N^{2}}\sum_{i\neq j}^{N}\delta_{x_{i}}(x)\delta_{x_{j}}(y)+\rho_{N,\hbar}^{(1)}(x)\rho_{N,\hbar}^{(1)}(y)\right.\\
&\left.-\frac{1}{N}\sum_{i=1}^{N}\delta_{x_{i}}(x)\rho_{N,\hbar}^{(1)}(y)-\rho_{N,\hbar}^{(1)}(x)\frac{1}{N}\sum_{j=1}^{N}\delta_{x_{j}}(y)}dxdy \rho_{N,\hbar}(X_{N})dX_{N}\\
=&\int G_{N}(x-y)\mu_{X_{N}}(dx)\mu_{X_{N}}(dy)\rho_{N,\hbar}(X_{N})dX_{N}-\frac{G_{N}(0)}{N}\int \rho_{N,\hbar}^{(1)}(x)dx,
\end{align*}
where
\begin{align*}
\mu_{X_{N}}(dx)=\frac{1}{N}\sum_{i=1}^{N}\delta_{x_{i}}(dx)-\rho_{N,\hbar}^{(1)}(x)dx.
\end{align*}
Then by \eqref{equ:G,square root} and \eqref{equ:equ:functional inequality,diagonal term}, we obtain
\begin{align*}
II=&\int \abs{G_{0,N}*\mu_{X_{N}}(z)}^{2} \rho_{N,\hbar}(X_{N})dzdX_{N}-\frac{G_{N}(0)}{N}\int \rho_{N,\hbar}^{(1)}(x)dx
\gtrsim- N^{3\eta-1},
\end{align*}
which completes the proof of estimate \eqref{equ:two body,one body}.
\end{proof}

To the end, we get into the proof of Proposition \ref{lemma:functional inequality,fdelta,V}.
\begin{proof}[\textbf{Proof of Proposition $\ref{lemma:functional inequality,fdelta,V}$}]
For estimate \eqref{equ:functional inequality,fdelta}, the approximation \eqref{equ:replacement,F=div,Fdelta} of $\wt{\mathcal{F}}_{\delta}(t)$ in Lemma \ref{lemma:replacement argument} gives
\begin{align*}
\wt{\mathcal{F}}_{\delta}=&-b_{0}\int \operatorname{div} u(x) G_{N}(x-y)
 \lrc{\frac{N-1}{N}\rho_{N,\hbar}^{(2)}(x,y)\right.\\
& \left. \quad \quad   -\rho_{N,\hbar}^{(1)}(x)\rho(y)-\rho(x)\rho_{N,\hbar}^{(1)}(y)+\rho(x)\rho(y)}dxdy \notag\\
&\pm O(N^{\be-1}\hbar^{-6} +N^{-\frac{\beta}{3}}\hbar^{-4}+N^{-\frac{\eta}{3}}\hbar^{-4}).
\end{align*}
Then by the functional inequality \eqref{equ:functional inequality,fdelta,G} in Lemma \ref{lemma:functional inequality,fdelta}, we get
\begin{align*}
\wt{\mathcal{F}}_{\delta}\leq& \n{\operatorname{div} u}_{L^{\infty}}\int G_{N}(x-y) \lrc{\frac{N-1}{N}\rho_{N,\hbar}^{(2)}(x,y)
\right.\\
&\left.-\rho_{N,\hbar}^{(1)}(x)\rho(y)-\rho(x)\rho_{N,\hbar}^{(1)}(y)+\rho(x)\rho(y)}dxdy\\
&+O(N^{\be-1}\hbar^{-6} +N^{-\frac{\beta}{3}}\hbar^{-4}+N^{-\frac{\eta}{3}}\hbar^{-4}+N^{-\eta}\hbar^{-2}+N^{3\eta-1}).
\end{align*}
Using the approximation \eqref{equ:replacement,F=1,Fdelta} of $\mathcal{F}_{\delta}$ in Lemma \ref{lemma:replacement argument}, we arrive at
\begin{align*}
\wt{\mathcal{F}}_{\delta}\leq& \n{\operatorname{div}u}_{L^{\infty}}\mathcal{F}_{\delta}+O(N^{\be-1}\hbar^{-6} +N^{-\frac{\beta}{3}}\hbar^{-4}+N^{-\frac{\eta}{3}}\hbar^{-4}+N^{-\eta}\hbar^{-2}+N^{3\eta-1})\\
\lesssim& \mathcal{F}_{\delta}+O(N^{\be-1}\hbar^{-6} +N^{-\frac{\beta}{3}}\hbar^{-4}+N^{-\frac{1}{10}}\hbar^{-4}),
\end{align*}
where in the last inequality we took $\eta=\frac{3}{10}$. Therefore, we complete the proof of the estimate \eqref{equ:functional inequality,fdelta}.

For the lower bound estimate \eqref{equ:lower bound,fdelta} on $\mathcal{F}_{\delta}$, we use the approximation \eqref{equ:replacement,F=1,Fdelta} of $\mathcal{F}_{\delta}$
 and estimate \eqref{equ:two body,one body} to obtain
\begin{align}
\mathcal{F}_{\delta}=&b_{0}\int G_{N}(x-y) \lrc{\frac{N-1}{N}\rho_{N,\hbar}^{(2)}(x,y)\right.\notag\\
&\quad \quad \left.-\rho_{N,\hbar}^{(1)}(x)\rho(y)-\rho(x)\rho_{N,\hbar}^{(1)}(y)+\rho(x)\rho(y)}dxdy\notag\\
&\pm O(N^{\be-1}\hbar^{-6} +N^{-\frac{\beta}{3}}\hbar^{-4}+N^{-\frac{\eta}{3}}\hbar^{-4})\notag\\
\geq&\int G_{N}(x-y) (\rho_{N,\hbar}^{(1)}(x)-\rho(x))(\rho_{N,\hbar}^{(1)}(y)-\rho(y))dxdy\label{equ:lower bound,fdelta,positive term}\\
&- O(N^{\be-1}\hbar^{-6} +N^{-\frac{\beta}{3}}\hbar^{-4}+N^{-\frac{\eta}{3}}\hbar^{-4}+N^{3\eta-1}).\notag
\end{align}
By \eqref{equ:G,square root}, we observe that the term on the r.h.s of \eqref{equ:lower bound,fdelta,positive term}
\begin{align*}
&\int G_{N}(x-y) (\rho_{N,\hbar}^{(1)}(x)-\rho(x))(\rho_{N,\hbar}^{(1)}(y)-\rho(y))dxdy\\
=&\int |G_{N,0}*(\rho_{N,\hbar}^{(1)}-\rho)(z)|^{2}dz\geq 0.
\end{align*}
Thus, we can discard this positive term and then take $\eta=\frac{3}{10}$ to get
\begin{align*}
\mathcal{F}_{\delta}\geq -O(N^{\be-1}\hbar^{-6} +N^{-\frac{\beta}{3}}\hbar^{-4}+N^{-\frac{1}{10}}\hbar^{-4}),
\end{align*}
which is the lower bound estimate \eqref{equ:lower bound,fdelta}.
\end{proof}

\section{Quantitative Strong Convergence of Quantum Densities}\label{section:Concluding the Strong Convergence of Quantum Densities}
In the section, using functional inequalities, we
prove the Gronwall's inequality for the modulated energy. Subsequently, with the quantitative convergence rate of the modulated energy, we further conclude the quantitative strong convergence of quantum mass and momentum densities. Notably, the $\delta$-type potential part is crucial in upgrading to the quantitative strong convergence, that is, in the case of only the Coulomb potential, one cannot deduce the strong convergence here.

Recall the modulated energy
\begin{align}
\mathcal{M}(t)=\mathcal{M}_{K}(t)+\mathcal{M}_{P}(t),
\end{align}
where the kinetic energy part is
\begin{align}
\mathcal{M}_{K}(t)=\int_{\mathbb{R}^{3N}}\vert\left(i \hbar \nabla_{x_{1}}-u(t,x_{1})\right)\psi_{N,\hbar}(t,X_N)\vert^2dX_N,
\end{align}
and the potential energy part is
\begin{align}
\mathcal{M}_{P}(t)=\mathcal{F}_{\delta}(t)
+\mathcal{F}_{c}(t).
\end{align}
From lower bound estimates \eqref{equ:lower bound,fdelta} on $\mathcal{F}_{\delta}(t)$ and \eqref{equ:lower bound,fc} on $\mathcal{F}_{c}(t)$,
we can add a small compensation such that
\begin{align}
\mathcal{F}_{\delta}(t)+\mathcal{F}_{c}(t)+r(N,\hbar)\geq 0,
\end{align}
where $r(N,\hbar)=C(N^{\be-1}\hbar^{-6}+N^{-\frac{\be}{3}}\hbar^{-4}+N^{-\frac{1}{10}}\hbar^{-4}).$
Thus, we introduce the positive modulated energy
\begin{align}\label{equ:modulated energy,modified}
\mathcal{M}^{+}(t)=&\mathcal{M}(t)+2r(N,\hbar) \geq r(N,\hbar)\geq 0.
\end{align}

We now provide a closed estimate for the positive modulated energy.
 \begin{proposition}
 For $t\in[0,T_{0}]$,
we have the differential inequality
\begin{align}\label{equ:upper bound for the evolution of modulated energy}
&\frac{d}{dt}\mathcal{M}^{+}(t)
\lesssim \mathcal{M}^{+}(t)+
\hbar^{2}.
\end{align}
Moreover, we conclude
\begin{equation}\label{equ:convergence of kinetic energy}
\int_{\mathbb{R}^{3N}}\vert\left(i \hbar \nabla_{x_{1}}-u(t,x_{1})\right)\psi_{N,\hbar}(t,X_N)\vert^2dX_N\lesssim
\mathcal{M}^{+}(0)+\hbar^{2},
\end{equation}
and
\begin{equation}\label{equ:convergence of mass density}
\int G_{N}(x-y) (\rho_{N,\hbar}^{(1)}(x)-\rho(x))(\rho_{N,\hbar}^{(1)}(y)-\rho(y))dxdy\lesssim \mathcal{M}^{+}(0)+\hbar^{2}.
\end{equation}

 \end{proposition}
 \begin{proof}
From the evolution of the modulated energy \eqref{equ:evolution of modulated energy}, we find that
\begin{align*}
&\frac{d}{dt}\mathcal{M}^{+}(t)\\
=& -\sum_{j,k=1}^{3}\int_{\mathbb{R}^{3N}}\lrs{\pa_{j}u^{k}+\pa_{k}u^{j}}(-i\hbar\partial_j\psi_{N,\hbar}-u^{j}\psi_{N,\hbar})
\overline{(-i\hbar\partial_k\psi_{N,\hbar}-u^{k}\psi_{N,\hbar})}dX_N\\
 &+\frac{\hbar^2}{2}\int_{\mathbb{R}^{3}}\Delta(\operatorname{div} u)(t,x_{1})\rho_{N,\hbar}^{(1)}(t,x_{1})dx_{1}
 +\wt{\mathcal{F}}_{\delta}(t)+
 \wt{\mathcal{F}}_{c}(t)\\
\lesssim &\n{\nabla u}_{L^{\infty}}\int_{\R^{3N}}|(i\hbar \nabla_{x_{1}}-u)\psi_{N,\hbar}|^{2}dX_{N}
+\hbar^{2}\n{\psi_{N,\hbar}}_{L^{2}}^{2}\n{\Delta \operatorname{div}u}_{L^{\infty}}+\wt{\mathcal{F}}_{\delta}(t)+
 \wt{\mathcal{F}}_{c}(t).
\end{align*}
By the functional inequalities $(\ref{equ:functional inequality,fdelta})$ on $\wt{\mathcal{F}}_{\delta}(t)$ and \eqref{equ:functional inequality,fc} on $\wt{\mathcal{F}}_{c}(t)$, we get
\begin{align}
\frac{d}{dt}\mathcal{M}^{+}(t)
\lesssim \mathcal{M}^{+}(t)+\hbar^{2}.
\end{align}
Then by Gronwall's inequality, we arrive at
\begin{align}\label{equ:modulated energy,estimate,final}
\mathcal{M}^{+}(t)\leq \exp(CT_{0})
\lrs{\mathcal{M}^{+}(0)+
\hbar^{2}t}
\lesssim& \mathcal{M}^{+}(0)+\hbar^{2}
\end{align}
for $t\in[0,T_{0}]$.

For the kinetic energy estimate \eqref{equ:convergence of kinetic energy}, by \eqref{equ:modulated energy,modified} and \eqref{equ:modulated energy,estimate,final} we have that
\begin{align*}
\int_{\mathbb{R}^{3N}}\vert\left(i \hbar \nabla_{x_{1}}-u(t,x_{1})\right)\psi_{N,\hbar}(t,X_N)\vert^2dX_N\leq  & \mathcal{M}^{+}(t)
\lesssim \mathcal{M}^{+}(0)+\hbar^{2},
\end{align*}
which completes the proof of \eqref{equ:convergence of kinetic energy}.

For the potential energy estimate \eqref{equ:convergence of mass density}, by \eqref{equ:two body,one body} in Lemma \ref{lemma:two body,one body} and \eqref{equ:replacement,F=1,Fdelta} in Lemma \ref{lemma:replacement argument}, we have
\begin{align*}
\int G_{N}(x-y) (\rho_{N,\hbar}^{(1)}(t,x)-\rho(t,x))(\rho_{N,\hbar}^{(1)}(t,y)-\rho(t,y))dxdy
\leq& \wt{\mathcal{F}}_{c}(t)+ r(N,\hbar)\\
\leq& \mathcal{M}^{+}(t)+ r(N,\hbar).
\end{align*}
Again by \eqref{equ:modulated energy,estimate,final}, we arrive at \eqref{equ:convergence of mass density}.
 \end{proof}

To the end, we get into the proof of Theorem $\ref{thm:main theorem}$.
\begin{proof}[\textbf{Proof of Theorem $\ref{thm:main theorem}$}]

\begin{flushleft}
\textbf{Convergence of the mass density $\rho_{N,\hbar}^{(1)}(t)$.}
\end{flushleft}

We decompose
\begin{align}
&\int G_{N}(x-y) (\rho_{N,\hbar}^{(1)}(t,x)-\rho(t,x))(\rho_{N,\hbar}^{(1)}(t,y)-\rho(t,y))dxdy\notag\\
=&\lra{(G_{N}-\delta)*(\rho_{N,\hbar}^{(1)}(t)-\rho(t)),\rho_{N,\hbar}^{(1)}(t)-\rho(t)}+\n{\rho_{N,\hbar}^{(1)}(t)-\rho(t)}_{L^{2}}^{2}. \label{equ:error term,potential part}
\end{align}

For the first term on the r.h.s of \eqref{equ:error term,potential part}, we use H\"{o}lder inequality and Lemma \ref{lemma:quantative estimate for identity approximation} to obtain
\begin{align*}
&\lra{(G_{N}-\delta)*(\rho_{N,\hbar}^{(1)}(t)-\rho(t)),\rho_{N,\hbar}^{(1)}(t)-\rho(t)}\\
\leq&  \n{(G_{N}-\delta)*(\rho_{N,\hbar}^{(1)}(t)-\rho(t))}_{L^{\frac{3}{2}}}\n{\rho_{N,\hbar}^{(1)}(t)-\rho(t)}_{L^{3}}\\
\lesssim & N^{-\eta}\lrs{\n{\lra{\nabla}\rho_{N,\hbar}^{(1)}(t)}_{L^{\frac{3}{2}}}+\n{\rho(t)}_{L^{\frac{3}{2}}}}\lrs{\n{\rho_{N,\hbar}^{(1)}(t)}_{L^{3}}
+\n{\rho(t)}_{L^{3}}}.
\end{align*}
Next, we estimate the terms $\n{\lra{\nabla}\rho_{N,\hbar}^{(1)}(t)}_{L^{\frac{3}{2}}}$ and $\n{\rho_{N,\hbar}^{(1)}(t)}_{L^{3}}$. By the Calder\'{o}n-Zygmund theory which implies that $\n{\lra{\nabla} f}_{L^{p}}\lesssim \n{\nabla f}_{L^{p}}+\n{f}_{L^{p}}$ for $1<p<\infty$, we get
\begin{align*}
\n{\lra{\nabla}\rho_{N,\hbar}^{(1)}(t)}_{L^{\frac{3}{2}}}\lesssim \n{\nabla\rho_{N,\hbar}^{(1)}(t)}_{L^{\frac{3}{2}}}
+\n{\rho_{N,\hbar}^{(1)}(t)}_{L^{\frac{3}{2}}}.
\end{align*}
By the Leibniz rule, Minkowski, H\"{o}lder, and Sobolev inequalities, we then obtain
\begin{align*}
\n{\nabla\rho_{N,\hbar}^{(1)}(t)}_{L^{\frac{3}{2}}}
\lesssim \n{\nabla_{x_{1}}\psi_{N,\hbar}(t)}_{L^{2}}\n{\psi_{N,\hbar}(t)}_{L^{2}L_{x_{1}}^{6}}\lesssim \n{\lra{\nabla_{x_{1}}}\psi_{N,\hbar}(t)}_{L^{2}}^{2}\lesssim \hbar^{-2},
\end{align*}
where in the last inequality we have used the $H^{1}$ energy bound \eqref{equ:energy conservation law} for $\psi_{N,\hbar}$.
Similarly, we also have
\begin{align*}
\n{\rho_{N,\hbar}^{(1)}(t)}_{L^{\frac{3}{2}}}\lesssim \n{\psi_{N,\hbar}(t)}_{L^{2}}\n{\psi_{N,\hbar}(t)}_{L^{2}L_{x_{1}}^{6}}\lesssim \n{\psi_{N,\hbar}(t)}_{L^{2}} \n{\lra{\nabla_{x_{1}}}\psi_{N,\hbar}(t)}_{L^{2}}\lesssim \hbar^{-1},
\end{align*}
and
\begin{align*}
\n{\rho_{N,\hbar}^{(1)}(t)}_{L^{3}}\leq \n{\psi_{N,\hbar}(t)}_{L^{2}L_{x_{1}}^{6}}^{2}\lesssim \n{\lra{\nabla_{x_{1}}}\psi_{N,\hbar}(t)}_{L^{2}}^{2}\lesssim \hbar^{-2}.
\end{align*}

With $\eta=\frac{3}{10}$, these bounds give that
\begin{align}\label{equ:error term,potential part,final}
\lra{(G_{N}-\delta)*(\rho_{N,\hbar}^{(1)}(t)-\rho(t)),\rho_{N,\hbar}^{(1)}(t)-\rho(t)}\lesssim N^{-\frac{3}{10}}\hbar^{-4}\lesssim r(N,\hbar).
\end{align}
Thus, combining \eqref{equ:error term,potential part}, \eqref{equ:error term,potential part,final} with \eqref{equ:convergence of mass density}, we arrive at
\begin{align}\label{equ:convergence of mass density,L2}
\n{\rho_{N,\hbar}^{(1)}(t)-\rho(t)}_{L^{2}}^{2}\lesssim &r(N,\hbar)+\mathcal{M}^{+}(0)+\hbar^{2}
\lesssim \mathcal{M}^{+}(0)+\hbar^{2},
\end{align}
where in the last inequality we have used that $r(N,\hbar)\leq \mathcal{M}^{+}(0)$.
\begin{flushleft}
\textbf{Convergence of the momentum density $J_{N,\hbar}^{(1)}(t)$.}
\end{flushleft}

Recall the momentum density
\begin{align*}
J_{N,\hbar}^{(1)}(t,x_{1})
=\hbar\int \operatorname{Im}(\overline{\psi_{N,\hbar}}  \nabla_{x_{1}}\psi_{N,\hbar})(t,X_N)dx_{2}\ccc dx_{N}.
\end{align*}
Then by the triangle and H\"{o}lder's inequalities, we have
\begin{align*}
&\n{J_{N,\hbar}^{(1)}(t)-(\rho u)(t)}_{L^{1}}\\
\leq &\n{J_{N,\hbar}^{(1)}(t)-(\rho_{N,\hbar}^{(1)}u)(t)}_{L^{1}}+\n{(\rho_{N,\hbar}^{(1)}u)(t)-(\rho u)(t)}_{L^{1}}\\
=&\n{\operatorname{Im}\lrs{\ol{\psi_{N,\hbar}(t)}\lrs{\hbar\nabla_{x_{1}} -iu(t)}\psi_{N,\hbar}(t)}}_{L^{1}}+\n{(\rho_{N,\hbar}^{(1)}u)(t)-(\rho u)(t)}_{L^{1}}\\
\leq& \n{\psi_{N,\hbar}(t)}_{L^{2}}\n{(i\hbar\nabla_{x_{1}}-u(t))\psi_{N,\hbar}(t)}_{L^{2}}+\n{u(t)}_{L^{2}}
\n{\rho_{N,\hbar}^{(1)}(t)-\rho(t)}_{L^{2}}\\
\lesssim& \mathcal{M}^{+}(0)+\hbar^{2},
\end{align*}
where in the last inequality we used the mass conservation, estimates $(\ref{equ:convergence of kinetic energy})$ and $(\ref{equ:convergence of mass density,L2})$.
\end{proof}

\noindent
\textbf{Acknowledgements.}
 X. Chen was supported in part by NSF grant DMS-2005469 and a Simons fellowship numbered 916862, S. Shen was supported in part by the Postdoctoral Science Foundation of China under Grant 2022M720263, and Z. Zhang was supported in part by NSF of China under Grant 12171010 and 12288101.

\appendix

\section{Sobolev Type Estimates}\label{section:Sobolev Type Estimates}
\begin{lemma}[\cite{CH22quantitative}, Lemma A.5]\label{lemma:quantative estimate for identity approximation}
Let $d=3$ and $W_{N}(x)=N^{3\be}V(N^{\be}x)-b_{0}\delta$, where $b_{0}=\int V(x)dx$. For any $0\leq s\leq 1$,
\begin{align}
\n{W_{N}*f}_{L^{p}}\leq C\n{\lra{x}V(x)}_{L^{1}} N^{-\be s}\n{\lra{\nabla}^{s}f}_{L^{p}}
\end{align}
for any $1<p<\wq$.
\end{lemma}

\begin{lemma}[\cite{ESY07}, Lemma A.3]\label{lemma:Sobolev type estimate,3d}
Let $d=3$ and $V_{N}(x)=N^{3\be}V(N^{\be}x)$. Then
\begin{align}
&V_{N}(x_{1}-x_{2})\leq C \n{V}_{L^{1}}(1-\Delta_{x_{1}})
(1-\Delta_{x_{2}}),\label{equ:Sobolev type estimate,3d,L1}\\
&V_{N}(x_{1}-x_{2})\leq C N^{\be}\n{V}_{L^{\frac{3}{2}}}(1-\Delta_{x_{1}})\label{equ:Sobolev type estimate,3d,L2},\\
&V_{N}(x_{1}-x_{2})\leq C N^{3\be}\n{V}_{L^{\wq}}.\label{equ:Sobolev type estimate,3d,Linffty}
\end{align}
\end{lemma}

\begin{lemma}\label{lemma:poincare type inequality,appendix}
Suppose that $f\in L^{1}$ such that
$$\int |f(x)||x|^{\frac{1}{2}}dx<\infty.$$
Let $f_{\ve}(x)=\ve^{3}f(\ve x)$ and $d_{0}=\int f dx$, then we have
\begin{align*}
&|\lra{(f_{\ve}(x-y)- d_{0}\delta(x-y))\vp,\psi}|\\
\lesssim&  \ve^{\theta}\lra{(1-\Delta_{x})(1-\Delta_{y})\vp,\vp}^{\frac{1}{2}}\lra{(1-\Delta_{x})(1-\Delta_{y})\psi,\psi}^{\frac{1}{2}}
\end{align*}
for $\theta\in (0,\frac{1}{2})$.
\end{lemma}
\begin{proof}
For the derivation of NLS, this Poincar$\acute{e}$ type inequality is usually used in the convergence part of the hierarchy method. See, for example, \cite{ESY07,ESY09,ESY10,KSS11}. For completeness, we here include a proof.
Without loss of generality, we might as well assume that $d_{0}=\int f dx=1$.
Switching to Fourier space, we observe that
\begin{align*}
&\lra{\vp,(f_{\ve}(x-y)-\delta(x-y))\psi}\\
=&\int dx dp d\xi_{1} d\xi_{2} \widehat{\vp}(\xi_{1},\xi_{2}) \overline{\widehat{\psi}}
(\xi_{1}+p,\xi_{2}-p)f(x)(e^{i\ve p\cdot x}-1).
\end{align*}
By using $|e^{ia}-1|\leq \min\lr{a,2}\leq 2a^{\theta}$ for $\theta\in (0,1)$ and  $|p|^{\theta}\leq \lra{\xi_{1}}^{\theta}+\lra{\xi_{1}+p}^{\theta}$, we have
\begin{align*}
&\lra{\vp,(f_{\ve}(x-y)-\delta(x-y))\psi}\\
\leq& 2\ve^{\theta}\int f(x)|x|^{\theta}dx \int dp d\xi_{1} d\xi_{2} \widehat{\vp}(\xi_{1},\xi_{2}) \overline{\widehat{\psi}}
(\xi_{1}+p,\xi_{2}-p)|p|^{\theta}\\
\leq& 2\ve^{\theta}\int f(x)|x|^{\theta}dx \int dp d\xi_{1} d\xi_{2} \widehat{\vp}(\xi_{1},\xi_{2}) \overline{\widehat{\psi}}
(\xi_{1}+p,\xi_{2}-p)\lrs{\lra{\xi_{1}}^{\theta}+\lra{\xi_{1}+p}^{\theta}}.
\end{align*}
It suffices to bound the term containing $\lra{\xi_{1}}^{\theta}$, as the term containing $\lra{\xi_{1}-p}^{\theta}$ can be estimated similarly. We rewrite
\begin{align*}
&\int dp d\xi_{1} d\xi_{2} \widehat{\vp}(\xi_{1},\xi_{2}) \overline{\widehat{\psi}}
(\xi_{1}+p,\xi_{2}-p) \lra{\xi_{1}}^{\theta}\\
=&\int dp d\xi_{1} d\xi_{2} \widehat{\vp}(\xi_{1},\xi_{2}) \overline{\widehat{\psi}}
(\xi_{1}+p,\xi_{2}-p)\frac{\lra{\xi_{1}}\lra{\xi_{2}}\lra{\xi_{1}+p}\lra{\xi_{2}-p}}
{\lra{\xi_{1}}^{1-\theta}\lra{\xi_{2}}\lra{\xi_{1}+p}\lra{\xi_{2}-p}}
\end{align*}
By Cauchy-Schwarz inequality,
\begin{align*}
&\int dp d\xi_{1} d\xi_{2} \widehat{\vp}(\xi_{1},\xi_{2}) \overline{\widehat{\psi}}
(\xi_{1}+p,\xi_{2}-p) \lra{\xi_{1}}^{\theta}\\
\leq&\lrc{\int dp d\xi_{1} d\xi_{2} \frac{\lra{\xi_{1}}^{2}\lra{\xi_{2}}^{2}}{\lra{\xi_{1}+p}^{2}\lra{\xi_{2}-p}^{2}}
|\widehat{\vp}(\xi_{1},\xi_{2})|^{2}}^{\frac{1}{2}}\\
&\cdot \lrc{\int dp d\xi_{1} d\xi_{2} \frac{\lra{\xi_{1}+p}^{2}\lra{\xi_{2}-p}^{2}}{\lra{\xi_{1}}^{2-2\theta}\lra{\xi_{2}}^{2}}
|\widehat{\psi}(\xi_{1}+p,\xi_{2}-p)|^{2}}^{\frac{1}{2}}\\
=&\lrc{\int dp d\xi_{1} d\xi_{2} \frac{\lra{\xi_{1}}^{2}\lra{\xi_{2}}^{2}}{\lra{\xi_{1}+p}^{2}\lra{\xi_{2}-p}^{2}}
|\widehat{\vp}(\xi_{1},\xi_{2})|^{2}}^{\frac{1}{2}}\\
&\cdot \lrc{\int dp d\xi_{1} d\xi_{2} \frac{\lra{\xi_{1}}^{2}\lra{\xi_{2}}^{2}}{\lra{\xi_{1}+p}^{2-2\theta}\lra{\xi_{2}-p}^{2}}
|\widehat{\psi}(\xi_{1},\xi_{2})|^{2}}^{\frac{1}{2}}\\
\lesssim &\lra{(1-\Delta_{1})(1-\Delta_{2})\vp,\vp}^{\frac{1}{2}}\lra{(1-\Delta_{1})(1-\Delta_{2})\psi,\psi}^{\frac{1}{2}}
\end{align*}
where in the last inequality we used that
\begin{align*}
\sup_{\xi_{1},\xi_{2}}\int \frac{1}{\lra{\xi_{1}-p}^{2-2\theta}\lra{\xi_{2}-p}^{2}}dp<\infty
\end{align*}
for all $0 \leq \theta<\frac{1}{2}$.
\end{proof}

\bibliographystyle{abbrv}

\bibliography{references}

\end{document}